\documentclass{article}
\usepackage[utf8]{inputenc}
\usepackage[colorlinks=true,linktocpage=true]{hyperref}
\usepackage{amsmath}
\usepackage{amssymb}
\usepackage{bbm}
\usepackage{amsthm}
\usepackage{enumerate}

\usepackage{amsmath}
\allowdisplaybreaks[4]
\usepackage{amssymb}
\usepackage{mathtools}
\usepackage{enumerate}
\usepackage{amsthm}
\usepackage{caption}
\usepackage{subcaption}
\usepackage{hyperref}
\hypersetup{
hidelinks
}

\usepackage[table]{xcolor} 
\usepackage[all]{xy}
\usepackage{tikz}

\newcommand\cA{{\mathcal A}}

\newcommand\cF{{\mathcal F}}

\newcommand\cH{{\mathcal H}}

\newcommand\cO{{\mathcal O}}

\newcommand\cU{{\mathcal U}}

\newcommand{\id}{{\rm id}}

\newcommand{\FS}{\mathrm{FS}}

\newcommand{\Vol}{\mathrm{Vol}}

\newcommand{\lct}{\mathrm{lct}}
\newcommand{\glct}{\mathrm{glct}}

\newcommand{\CC}{\mathbb {C}}

\newcommand{\MM}{{\mathbb M}}
\newcommand{\NN}{{\mathbb N}}
\newcommand{\PP}{{\mathbb P}}
\newcommand{\QQ}{{\mathbb Q}}
\newcommand{\RR}{{\mathbb R}}
\newcommand{\ZZ}{{\mathbb Z}}

\DeclareMathOperator{\Int}{int}
\DeclareMathOperator{\Aut}{Aut}
\DeclareMathOperator{\Ver}{Ver}
\DeclareMathOperator{\Conv}{co}
\DeclareMathOperator{\co}{co}
\DeclareMathOperator{\Span}{span}

\theoremstyle{plain}
\newtheorem{theorem}{Theorem}[section]
\newtheorem{proposition}[theorem]{Proposition}
\newtheorem{prop}[theorem]{Proposition}

\newtheorem{lemma}[theorem]{Lemma}

\newtheorem{claim}[theorem]{Claim}
\newtheorem{corollary}[theorem]{Corollary}

\newtheorem{conjecture}[theorem]{Conjecture}

\newtheorem{problem}[theorem]{Problem}

\theoremstyle{definition}
\newtheorem{definition}[theorem]{Definition}
\newtheorem{remark}[theorem]{Remark}
\newtheorem{example}[theorem]{Example}

\def\o{\omega}
\def\K{K\"ahler }
\def\KE{K\"ahler--Einstein }
\newcommand{\beq}{\begin{equation}}
\newcommand{\eeq}{\end{equation}}
\newcommand{\bpf}{\begin{proof}}
\newcommand{\epf}{\end{proof}}
\newcommand{\bdefn}{\begin{definition}}
\newcommand{\edefn}{\end{definition}}
\newcommand{\bremark}{\begin{remark}}
\newcommand{\eremark}{\end{remark}}
\newcommand{\bconj}{\begin{conjecture}}
\newcommand{\econj}{\end{conjecture}}
\newcommand{\bcor}{\begin{corollary}}
\newcommand{\ecor}{\end{corollary}}
\newcommand{\blem}{\begin{lemma}}
\newcommand{\elem}{\end{lemma}}
\newcommand{\bclaim}{\begin{claim}}
\newcommand{\eclaim}{\end{claim}}
\newcommand{\bprob}{\begin{problem}}
\newcommand{\eprob}{\end{problem}}
\newcommand{\bprop}{\begin{proposition}}
\newcommand{\eprop}{\end{proposition}}
\newcommand{\bthm}{\begin{theorem}}
\newcommand{\ethm}{\end{theorem}}
\def\lb#1{\label{#1}}
\def\ra{\rightarrow}
\def\ran{\rangle}
\def\lan{\langle}
\def\q{\quad}
\def\vp{\varphi}
\def\i{\sqrt{-1}}
\def\ddbar{\partial\overline{\partial}}
\def\disp{\displaystyle}
\def\texts{\textstyle}
\def\eps{\epsilon}

\title{Tian's stabilization problem for toric Fanos}
\author{Chenzi Jin, Yanir A. Rubinstein }
\date{25 March 2024}

\begin{document}

\maketitle

\centerline{\it In memory of Eugenio Calabi}

\bigskip

\begin{abstract}
In 1988, Tian posed the stabilization problem for equivariant global log canonical
thresholds. We solve it in the case of toric Fano manifolds. This is 
the first general result on Tian's problem.
A key new estimate involves expressing complex singularity exponents
associated to orbits of a group action in terms of support and gauge functions from
convex geometry.
These techniques also yield a resolution of another conjecture of Tian from 2012
on more general thresholds associated to Grassmannians of plurianticanonical
series. 
\end{abstract}

\tableofcontents

\section{Introduction}

This article is the first in a series in which we study asymptotics of 
invariants related to 
existence of canonical metrics on \K manifolds. 

In the present article we focus on Tian's
$\alpha_{k,G}$- and $\alpha_{k,m,G}$-invariants
that were defined 
in 1988 and 1991, and to a large extent are still rather mysterious.
The presence of the compact symmetry group $G$ is a major source of difficulty and new ideas are needed
here as these invariants have not been previously systematically
studied or computed. 
In particular, we provide a formula for such invariants, valid for all toric Fano manifolds, leading to a resolution of Tian's
stabilization conjecture in this setting, which is also
the first general result on Tian's conjecture.

A sequel to this article deals with the 
$\delta_k$-invariants of Fujita--Odaka
(also
called $k$-th stability thresholds)
on toric Fano manifolds.
It turns out that for these invariants
there is a quantitative dichotomy regarding stabilization, and when stabilization
fails we derive their complete asymptotic expansion.

\subsection{Tian's stabilization problem}
\lb{TianStabilizationSubSec}

Let $(X,L,\o)$ be a polarized \K manifold of dimension $n$ with $L$ 
a very ample line bundle over $X$, and $\omega$ a \K form
representing $c_1(L)$.
The space 
\beq
\lb{HLEq}
\cH_L:=
\{\vp \,\,:\,\, \omega_\vp:=\o+\i\ddbar\vp>0\}\subset C^\infty(X)
\eeq
of \K potentials of metrics cohomologous to $\omega$ 
was introduced by Calabi in a short visionary talk in the Joint AMS--MAA Annual Meetings held
at Johns Hopkins University in December, 1953 \cite{Calabi}.
In a groundbreaking article
some 35 years later, Tian proved
(motivated by a question of Yau \cite[p. 139]{Yau2})
that $\cH_L$ is approximated
(or ``quantized")
in the $C^2$ sense 
by the finite-dimensional
spaces 
$
\cH_k
$
consisting of pull-backs of Fubini--Study metrics on $\PP(H^0(X,L^k)^*)$ 
under all possible Kodaira embeddings induced by $H^0(X,L^k)$
\cite{Tian90}.
A decade later this was improved to a complete asymptotic expansion
\cite{Catlin,Zel}
and so the $\cH_k$ can be considered as the Taylor 
(or Fourier, depending on the point of view) expansion of $\cH_L$.
The theme  
that holomorphic
and other 
invariants associated to $X$ and $\cH$
may be quantized using the 
spaces $\cH_k$
has dominated K\"ahler geometry for the last 
35 years.

In the 1980's, Futaki's invariant
was a new obstruction for the existence of \KE metrics, but
there were no invariants that guaranteed existence. At best,
there were constructions that utilized symmetry to reduce the \KE equation to a simpler equation that could be solved and lead to specific examples
(another theme pioneered by Calabi). 
Given a maximal compact subgroup $G$ of the automorphism
group $\Aut X$, Tian introduced
the invariant
$$
\alpha_G:=\sup
\left\{c>0\,:\,\sup_{\varphi\in \cH^G}\int_Xe^{-c(\varphi-\sup\varphi)}\omega^n<\infty\right\}
$$
(where $\cH^G$ denotes the $G$-invariant elements of $\cH_L$),
and its quantized version (Definition \ref{alphakGhmuDef})
$$
\alpha_{k,G},
$$
computed over the $G$-invariant elements of $\cH_k$,
and obtained the sufficient condition  
\beq
\lb{TianBoundEq}
\alpha_G>\frac n{n+1}
\eeq
for the existence of a \KE metric
when $L=-K_X$ (which we will henceforth assume unless otherwise stated)
\cite[Theorem 4.1]{Tian87}.
Initially, the main interest in 
the invariants $\alpha_G$ was as
the first systematic tool for constructing \KE metrics on Fano
manifolds, but later it was also conjectured by Cheltsov
and established by Demailly that $\alpha_G$ actually coincides
with the $G$-equivariant global log canonical threshold from
algebraic geometry \cite{CS08}. 
Since $\cH_k\subset\cH_L$ it follows that  
$\alpha_G\le\inf_k\alpha_{k,G}$ \cite[p. 128]{Tian90},
yet this does not help obtain \eqref{TianBoundEq}.
Instead, Tian posed the following difficult question that would reduce the computation of 
the invariant $\alpha_G$ from the infinite-dimensional space $\cH_L$
to a finite-dimensional one $\cH_k$, and establish
a highly non-trivial relation between the different $\cH_k$'s
\cite[Question 1]{Tian90},\cite{Tian88}:

\bprob
\lb{TianProb}
Let $X$  be Fano and $L=-K_X$, and let $G$
be a maximal compact subgroup $G\subset\Aut X$.
Is $\alpha_{k,G}=\alpha_G$ for all 
sufficiently large $k\in\NN$?
\eprob

It is interesting to note that
a slight variation
of Tian's $\alpha$- and $\alpha_k$-invariants
turned out to lead about three decades
later \cite{Zhang,RTZ} to the very closely
related $\delta$- and $\delta_k$-invariants
of Fujita--Odaka \cite{FO} that
are in turn a slight (and ingenius) variation on global log canonical thresholds,
and turn out to
essentially characterize
the existence of \KE metrics. 
We return to these invariants in a sequel \cite{JR-delta-toric}.

\subsection{A Demailly type identity
in the presence
of symmetry}

Another (easier) problem motivating this article concerns the by-now-classical
relation between Tian's (holomorphic) invariants and 
the (algebraic) global log canonical thresholds. The relationship
was first conjectured by Cheltsov and proved by Demailly and Shi \cite{CS08,Shi10}.
However, so far, this relationship has only been shown for 
the $\alpha$- and $\alpha_k$-invariants, or for the $\alpha_G$-invariant
(see \cite[Theorem A.3]{CS08}, \cite[Proposition 2.1]{Shi10}, \cite[(A.1)]{CS08}, respectively), 
and not for the more subtle invariants $\alpha_{k,G}$.
Inspired by Demailly, we introduce
(Definition \ref{glctkDef})
the $k$-th $G$-equivariant global
log canonical threshold
$$
\glct_{k,G}
$$
as an algebraic counterpart of Tian's
$\alpha_{k,G}$ 
(Definition \ref{alphakGhmuDef}).
A natural question is:

\bprob
\lb{DemaillyProb}
Let $X$  be Fano and $L=-K_X$, and let $G$
be a compact subgroup $G\subset\Aut X$.
Is
$
\glct_{k,G}
=\alpha_{k,G}?
$
\eprob

\subsection{Results}
\lb{ResultsSubSec}

In this article, we resolve both Problems \ref{TianProb} and \ref{DemaillyProb} 
in the toric setting. 
It is perhaps not well-known, but Calabi was interested in toric geometry
and computed certain geodesics in $\cH_L$ in the toric setting, although
he never published the result \cite{Calabi-abstract}.

As standard, we allow the slightly more flexible situation
of any (and not just a maximal) compact subgroup of the normalizer
$$N((\CC^*)^n)$$
of the complex torus $(\CC^*)^n$ in $\Aut X$.
We refer the reader to \S\ref{ToricSec} where these and other 
toric notation is set-up carefully.
Denote 
by 
\beq
\lb{AutPEq}
\Aut P\subseteq GL(M)\cong GL(n,\ZZ)
\eeq
the subgroup of the automorphism group of the lattice $M$ that leaves the polytope $P$ \eqref{PQcircDef} invariant.
It is necessarily a finite group (see \S\ref{ToricSec}).
In fact, $\Aut P$ is the quotient of the normalizer $N((\CC^*)^n)$ of the
complex torus $(\CC^*)^n$ in $\Aut X$ by $(\CC^*)^n$, so that 
$N((\CC^*)^n)$  consists
of finitely many components each isomorphic to a complex torus \cite[Proposition 3.1]{BS99}.
For $H\subset\Aut P$, 
let 
\beq
\lb{GHEq}
G(H):=H\ltimes(S^1)^n\subset N((\CC^*)^n)\subset\Aut X
\eeq
denote the compact group generated by $H$ and $(S^1)^n$
(the latter is the maximal compact subgroup of the complex torus $(\CC^*)^n$).

Our first result resolves Problem \ref{DemaillyProb}
in this generality.

\bprop
\lb{DemaillyProblemProp}
Let $X$  be toric Fano and $L=-K_X$. Let $P\subset M_\RR$   
(see (\ref{NRMREq}), (\ref{PQcircDef})) be the polytope associated to $(X,-K_X)$, 
let
    $H\subset\Aut P$, and let $G(H)$ be as in \eqref{GHEq}.
Then
$
\glct_{k,G(H)}
=\alpha_{k,G(H)}.
$
\eprop

Using this result, and several new estimates, we can resolve Tian's 
Problem \ref{TianProb} in the toric setting in a surprisingly strong sense,
showing that equality holds
for {\it all} $k\in\NN$. We also allow for all groups $G(H)$
(and not just the maximal toric one $G(\Aut P)$).

To state the precise result we introduce some more notation. 
For $H\subset\Aut P$, denote by 
\beq\lb{PHEq}
P^H:=\big\{y\in P\,:\, h.y=y, \quad \forall h\in H\big\}\subset P\subset M_\RR
\eeq
the fixed-point set of $H$ in $P$, and
let 
\begin{equation}\label{definition of p}
    \pi_H:=\frac{1}{|H|}\sum_{\eta\in H}\eta\in
    \hbox{End}(M_\QQ)
\end{equation}
be the map that takes a point in $M_\RR$ to the average of its $H$-orbit. 
Note that $\pi_H$ is a projection map (see \S\ref{GroupOrbSubSec} for details).
 
\begin{theorem}\label{alpha_k,Gformula}
    Let $X$ be toric Fano associated to a fan $\Delta$ whose
    rays are generated by primitive elements $v_i$ in the lattice $N$ dual to $M$. 
    Let $P\subset M_\RR$ (see (\ref{NRMREq}), (\ref{PQcircDef})) be 
the polytope associated to $(X,-K_X)$, 
let
    $H\subset\Aut P$, and let $G(H)$ be as in \eqref{GHEq}.
Then
    for any $k\in\NN$, 
    \begin{align}
        \alpha_{k,G(H)}&=\sup\left\{c\in(0,1)\,:\,-\frac{c}{1-c}P^H\subset P\right\}\nonumber\\
        &=\min_{u\in\Ver P^H}\frac{1}{
        \max_i\langle u,v_i\rangle+1}    \lb{MainTheoremEq}\\
        &=\min_{u\in \pi_H(\Ver P)}\frac{1}{\max_i\langle u,v_i\rangle+1},\nonumber
    \end{align}
     where $P^H$ and $\pi_H$ are defined in 
     \eqref{PHEq}--\eqref{definition of p} and $\Ver(\,\cdot\,)$ denotes the vertex set of a polytope. 
    In particular, $\alpha_{k,G(H)}$ is independent of $k\in\NN$ and is equal to $\alpha_{G(H)}$.
\end{theorem}

There are a few new ingredients in the proof of Theorem \ref{alpha_k,Gformula}. The first is  a useful formula for the spaces $\cH_{k}^{G(H)}$ in terms
of the $H$-orbits of the finite group action (Lemma 
\ref{HkGHtoricLem}). This together with a trick that amounts
to estimating the singularities associated to a basis of sections
in terms of the finite group action orbit of a section yields
a useful formula for $\alpha_{k,G(H)}$ 
(Proposition \ref{analytic alpha_k,G}) as well as 
the equality $\alpha_{k,G(H)}=\glct_{k,G(H)}$, i.e.,
a solution to Problem \ref{DemaillyProb} 
(Proposition \ref{DemaillyProblemProp}).
These then yield a corresponding useful formula for
$\alpha_{G(H)}$ (Corollary \ref{alpha orbit exp cor}). 
One may prove using the results of 
\S\ref{NaturalvolumeformSec}--\S\ref{DemaillySec}
that $\alpha_{k,G(H)}\ge \alpha_{k\ell,G(H)}$  
for any fixed $k$ and all $\ell\in\NN$ (Proposition \ref{weakmonotoneProp}).
In Proposition \ref{divisibleProp} it is shown that there is a special $k_0$
for which $\alpha_{k_0\ell,G(H)}=\alpha_{G(H)}$ for all $\ell\in\NN$, as well as 
observed that this does not seem to imply Tian's conjecture (Remark
\ref{DivisibleRemark}). 
Finally, key new estimates occur in \S\ref{OrbitsEstimateSec}. First, we show that rather general
complex singularity exponents associated to collections of toric monomials are independent of $k$ (Proposition \ref{generalizeSong}).
The proof of this uses a new observation about the
relation between the support functions of collections of lattice points associated to the toric monomials and complex singularity exponents. We then apply this to our $G(H)$-invariant setting, using
the aforementioned expression of $\cH_k^{G(H)}$ and a reduction lemma
to the $H$-invariant subspace (Lemma \ref{Orbit lct simplified}),
to conclude the proof of Theorem \ref{alpha_k,Gformula}.

It is perhaps of some interest to include here a rather immediate application
of this circle of ideas to a slightly more technical set of invariants, 
also introduced by Tian, that we call Tian's Grassmannian $\alpha$-invariants.
These invariants are defined a little differently, algebraically,
and are denoted $\alpha_{k,m}$ or $\alpha_{k,m,G}$  (Definition \ref{alphakmGDef}).
At least in the non-equivariant setting (as well as in the torus-equivariant setting, see Remark \ref{GenAlphakmRem}) these can be considered as generalizations of the 
$\alpha_k$ as $\alpha_k=\alpha_{k,1}$ \cite{Shi10,CS08}. The invariants
$\alpha_{k,2}$ were used by Tian implicitly in his proof of Calabi's conjecture
for del Pezzo surfaces \cite[Appendix A]{Tian90} (cf. \cite[Theorem 6.1]{Tian91}) to overcome the most difficult case (of a cubic surface with an Eckardt point) where equality
holds in \eqref{TianBoundEq}, and this was improved by Shi to
$\alpha_{k,2}>2/3=\alpha_{k,1}$ in that case 
\cite[Theorem 1.3]{Shi10},\cite{Cheltsov08}.

Tian also posed a stabilization conjecture for these invariants in 2012 \cite[Conjecture 5.3]{Tian12}:

\bconj
\lb{TianGenConj}
Let $X$ be Fano.
Fix $m\in\NN$. For sufficiently large $k$, $\alpha_{k,m,G}$ is constant.
\econj

We completely resolve Conjecture \ref{TianGenConj} in the toric setting.
Theorem \ref{alpha_k,Gformula} resolved Problem \ref{TianProb}
in the affirmative (corresponding to the case $m=1$ of Conjecture \ref{TianGenConj}).
For $m\ge2$, Conjecture \ref{TianGenConj}  turn out
to be only partially true as determined by a 
novel convex geometric obstruction we introduce:
\beq
\label{starPEq}\tag{$\ast_P$}
\|\,\cdot\,\|_{-P}\Big|_{P\setminus\Ver P} 
<\;
\max_{P}\|\,\cdot\,\|_{-P}.
\eeq
Note that this condition depends only on $P$ (and not on $m,k$).
The condition \eqref{starPEq} means that 
the function $\|\cdot\|_{-P}$ on $P$ achieves its maximum 
only at the vertices of $P$, i.e.,
\beq
\label{starPequivEq}\tag{$\ast_P$}
\hbox{\rm argmax}_{P}\|\,\cdot\,\|_{-P}\subset \Ver P.
\eeq
When \eqref{starPEq} fails the
maximum is  achieved also at some point that is not a vertex of $P$.

\bthm    
\lb{SecondMainThm}
Let $X$ be toric Fano with associated polytope $P$ (\ref{PDef2Eq}). 
Conjecture \ref{TianGenConj} holds if and only if \eqref{starPEq} fails.
More precisely, if \eqref{starPEq} holds,
    \beq
    \lb{TianGenConjFailsEq}
       \alpha_{k,m,\left(S^1\right)^n}>\alpha, \q
           \hbox{for $k\in\NN$ and $m\in\NN\setminus\{1\}$},
    \eeq
otherwise
    \beq
    \lb{TianGenConjHoldsEq}
       \alpha_{k,m,\left(S^1\right)^n}=\alpha, 
       \q 
       \hbox{for $m\in\NN$ and for sufficiently large $k\in\NN$}.
    \eeq
\ethm

Theorem \ref{SecondMainThm} is proven in
\S\ref{ProofThm1.6SubSec} where we also explain the intuition behind it
(see also Examples \ref{alphakmGHExample} and \ref{alphakmSymmExample}).
For now, let us elucidate the condition \eqref{starPEq} a bit.
The level set $\{\|\,\cdot\,\|_{-P}=\lambda\}$ is the dilation 
$\lambda\partial(-P)$, and $\{\|\,\cdot\,\|_{-P}=\max_P\|\cdot\|_{-P}\}$ is the largest dilation that intersects $P$ (by Lemma \ref{norm of set}).
Thus, condition \eqref{starPEq} states that $P$ intersects $\max_P\|\cdot\|_{-P}\partial(-P)$ only at vertices. When \eqref{starPEq} fails, convexity arguments
 show the intersection will contain a positive-dimensional face of $P$.
See Figure \ref{*P example} for two examples. 

\begin{figure}
    \centering
    \begin{subfigure}{0.4\textwidth}
        \centering
        \begin{tikzpicture}
            \draw[->,gray](-2.5,0)--(1.5,0)node[below]{$x$};
            \draw[->,gray](0,-2.5)--(0,1.5)node[left]{$y$};
            \draw(0,0)node[below right]{$0$};
            \draw(-.5,-.5)--(1,-.5)--(-.5,1)--cycle;
            \draw[dashed](1,1)--(-2,1)--(1,-2)--cycle;
        \end{tikzpicture}   
    \end{subfigure}
    \begin{subfigure}{0.4\textwidth}
        \centering
        \begin{tikzpicture}
            \draw[->,gray](-3,0)--(2,0)node[below]{$x$};
            \draw[->,gray](0,-2)--(0,2)node[left]{$y$};
            \draw(0,0)node[below right]{$0$};
            \draw(-.5,-.5)rectangle(1,.5);
            \draw[dashed](-2,-1)rectangle(1,1);
        \end{tikzpicture}    
    \end{subfigure}
    \caption{The polytope $P$ (solid line) and the level set $\{\|\cdot\|_{-P}=\max_P\|\cdot\|_{-P}\}$ (dashed line). For $P=\Conv\{(-1,-1),(2,-1),(-1,2)\}$, the maximum is only attained at the vertices of $P$. In particular, \eqref{starPEq} holds. For $P=[-1,2]\times[-1,-1]$, the maximum is attained on the line segment $\{2\}\times[-1,-1]$. In particular, \eqref{starPEq} does not hold.}\label{*P example}
\end{figure}
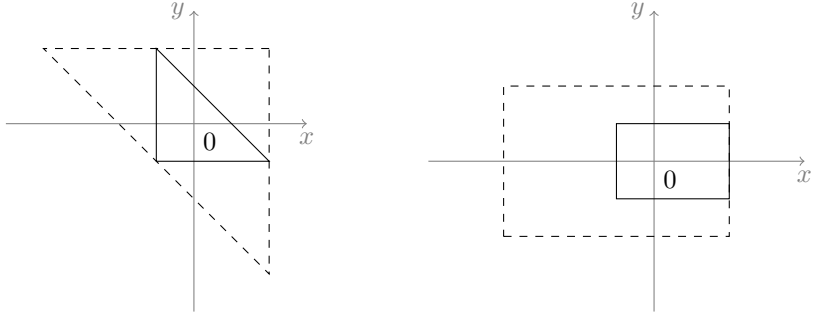

\bigskip
\noindent
{\bf Relation to earlier works.}
Theorem \ref{alpha_k,Gformula} strengthens and clarifies work of Song \cite{Song}
and Li--Zhu \cite{LiZhu}.
Song obtained a formula for $\alpha_G$ 
but not for $\alpha_{k,G}$. 
In particular, 
there seems to be a gap
in the proof
of \cite[Theorem 1.2]{Song}
that claims that $\alpha_G=\alpha_{k,G}$
for all sufficiently large $k$.
This claim relies on proving  
that for some $k_0\in\NN$
and all $\ell\in\NN$, 
$\alpha_G=\alpha_{k_0\ell,G}$
and then invoking that $\alpha_{k,G}$ is eventually monotone in $k$, and
hence must be independent of $k$ for sufficiently large $k$.
Unfortunately, the proof of monotonicity
is omitted from \cite[p. 1257, line 7]{Song}, 
and it appears to be difficult to reproduce. 
It seems that such monotonicity is not currently known (cf.
Remark
\ref{DivisibleRemark}).
Indeed, there is no obvious
relationship between the various $\mathcal H^G_k$ coming from different
Kodaira embeddings.  
As noted above, Song showed (for $H=\Aut P$) that $\alpha_{G(H)}=
\alpha_{k_0\ell,G(H)}$ for  some $k_0\in\NN$ and all $\ell\in\NN$.
Li--Zhu showed the same identity (essentially for $H=\{\id\}$) for a group compactification of a reductive complex Lie group and also claimed,
similarly to Song, that this implies eventual constancy in $k$ in that setting
\cite[Theorem 1.3, p. 233]{LiZhu}. Unfortunately, also they do not provide
a proof of the needed eventual monotonicity or constancy.
In the non-equivariant
setting, and for rather general Fano varieties for which $\alpha\le1$,
Birkar showed the deep result that  $\alpha=\alpha_{k_0\ell}$ for  some $k_0\in\NN$ and all $\ell\in\NN$ \cite[Theorem 1.7]{Birkar2022}. 
However, also this result does not imply Tian stabilization due
to the aforementioned unknown monotonicity. 
Thus, Theorem \ref{alpha_k,Gformula} seems to be the first general result
on Tian's stabilization Problem \ref{TianProb}.

Similarly, Theorem \ref{SecondMainThm} seems to be the first general result
on Conjecture \ref{TianGenConj}. 
Indeed, Li--Zhu showed the same type of result
Song obtained in the $m=1$ setting,
i.e., that $\alpha_{k_0\ell,m,(S^1)^n}=\alpha_{(S^1)^n}$ under a condition
depending on $k_0$ and $m$, and hence
different from our \eqref{starPEq} (with the minor caveat that their
statement as written \cite[Theorem 1.4]{LiZhu} is incorrect, 
though can be easily fixed by replacing ``facet" by ``face", see 
\S\ref{ProofThm1.6SubSec}).
However, again, due to the lack of monotonicity, they do not obtain
a resolution of Conjecture \ref{TianGenConj} though they do obtain 
the first counterexamples to it when $m\ge2$.

Combining Theorem \ref{alpha_k,Gformula} and Demailly's theorem \cite[(A.1)]{CS08}
also recovers 
Song's formula for the $\alpha_{G(\Aut P)}$
(that itself generalized Batyrev--Selivanova's formula that 
$\alpha_{G(\Aut P)}=1$ whenever $P^{\Aut P}=\{0\}$ (recall \eqref{PHEq}) \cite[Theorem 1.1, p. 233]{BS99}).
Cheltsov--Shramov claimed a more general formula for 
$\alpha_{H}$ 
(i.e., without the real torus symmetry included in $G(H)$, recall \eqref{GHEq})
however (as kindly pointed out to us by I. Cheltsov) 
there is an error in the proof of \cite[Lemma 5.1]{CS08}
as the toric degeneration used there need not respect the $H$-invariance. 
Further generalizations of Song's formula for $\alpha_G$ to general polarizations and group compactifications
are due to Delcroix \cite{Delcroix15,Delcroix17}
and Li--Shi--Yao
\cite{LSY15}
and our methods should generalize
to those settings as well as to the setting of log toric Fano pairs and edge singularities
\cite[\S6--7]{CR15}. 

Finally, it is also worth mentioning that Tian also posed more general
conjectures \cite[Conjecture 5.4]{Tian12}
for general polarizations (i.e., $L$ not being $-K_X$) 
for which there are already some counterexamples \cite{ACS}.

\bigskip
\noindent
{\bf Organization.}
Section \ref{ToricSec} sets up the necessary notation concerning
toric varieties and convex analysis. Section
\ref{NaturalvolumeformSec}
constructs natural equivariant reference Hermitian metrics and volume forms.
Proposition \ref{DemaillyProblemProp} is proved in \S\ref{DemaillyProbSubSec}.
Section \ref{DivisibleSec} explains a trick that allows to 
deal with divisible $k\in\NN$ but also highlights the difficulties
in dealing with general $k$.
Theorem \ref{alpha_k,Gformula} is proved in \S\ref{OrbitsEstimateSec}.
Theorem \ref{SecondMainThm} is proved in \S\ref{alphakmSec}.
We conclude with examples in \S\ref{ExamplesSec}.

\bigskip
\noindent
{\bf Acknowledgments.} 
Thanks to C. Birkar, H. Blum, I. Cheltsov, and K. Zhang for helpful references.
Research supported in part 
by NSF grants DMS-1906370,2204347, BSF
grant 2020329, and an
Ann G. Wylie Dissertation Fellowship.

\section{Toric and convex analysis set-up}
\lb{ToricSec}

\subsection{Notions from convexity}

Consider an $n$-dimensional real vector space $V\cong\RR^n$
and let $V^*\cong\RR^n$ denote its dual with the pairing denoted
by $\lan\,\cdot\,,\,\cdot\,\ran$.
Given a set $A\subset V$, denote by   
$$
A^\circ=\{y\in V^*\,:\, \langle x,y\rangle\le1,
\,\, \forall x\in A\}
$$
the polar of $A$ \cite[p. 125]{Rock}, and by
\beq
\lb{coA}
\co A
\eeq
the convex hull of $A$ \cite[p. 12]{Rock}. For a finite set \cite[Theorem 2.3]{Rock},
\beq
\lb{coAFiniteSet}
\co \{p_1,\ldots,p_\ell\}=
\Big\{\texts
\sum_{i=1}^\ell\lambda_ip_i\,:\,  \sum_{i=1}^\ell\lambda_i=1,
\; \lambda\in[0,1]^\ell
\Big\}.
\eeq

Also, set
$$
-K:=\{-x\,:\, x\in K\}.
$$
Note $(-K)^\circ=-K^\circ$. Also, $K^\circ=(\co K)^\circ\subset V^*$ whenever $K\subset V$.
The polar can also be described via the support function
$h_K:V^*\ra \RR$,
\beq
\lb{hKEq}
h_K(y):=\sup_{x\in K}\langle x,y\rangle, \quad y\in V^*\cong\RR^n,
\eeq
by
$K^\circ=\{h_K\le 1\}$.

A dual notion to the support function is the near-norm function
    \begin{equation}\label{normdef}
        \left\|x\right\|_K:=\inf\left\{t\geq0\,:\,x\in tK\right\},
    \end{equation}
associated to any compact convex set $K$ with $0\in\Int K$.
    Note that \cite[Corollary 14.5]{Rock}
\beq
\lb{hKpolarEq}
\left\|\,\cdot\,\right\|_K=h_{K^\circ}.
\eeq
Note that $\left\|\,\cdot\,\right\|_K$ is a norm when $K$
is centrally symmetric (i.e., $K=-K$), otherwise it is only a near-norm
in the sense that it satisfies all the properties of a norm
but is only \hbox{$\RR_{+}$-homogeneous:} $\left\|\lambda x\right\|_K=\left\|x\right\|_K$
for $\lambda\in\RR_+$ (and not fully $\RR$-homogeneous).
    The infimum in \eqref{normdef} is achieved:  for a minimizing sequence
    $\{t_i\}$, $\frac{x}{t_i}\in K$ for any $i$, so 
\beq
\lb{infachievedEq}
x\in\|x\|_KK
\eeq
since $K$ is closed.

\blem
    Let $V$ be an $\RR$-vector space. Consider the polytope
    $$
        A=\bigcap_{j=1}^d\left\{x\in V\,:\,\left\langle x,v_j\right\rangle\leq1\right\},
    $$
    where $v_j\in V^*$. Then,
    $$
        \left\|x\right\|_A=\max_{1\leq j\leq d}\left\langle x,v_j\right\rangle.
    $$
\elem
\begin{proof}
    Notice that $x\in tA$ if and only if for any $1\leq j\leq d$, $\langle x,v_j\rangle\leq t$. Thus,
    $$
        \left\|x\right\|_A=\inf\left\{t\geq0\,:\,\max_{1\leq j\leq d}\left\langle x,v_j\right\rangle\leq t\right\}=\max_{1\leq j\leq d}\left\langle x,v_j\right\rangle.
    $$
Alternatively, observe that 
$A=\{v_1,\ldots,v_d\}^\circ$ and 
$A^\circ=(\{v_1,\ldots,v_d\}^\circ)^\circ=\co\{v_1,\ldots,v_d\}$
\cite[Theorem 14.5]{Rock} and then use \eqref{hcAcocAEq} and \eqref{hKpolarEq}.
\end{proof}

\subsection{Toric algebra}

Consider a lattice of rank $n$ and its dual lattice 
\beq
\lb{NMHomEq}
N, \quad
M:=N^*:=
\mathrm{Hom}(N,\ZZ).
\eeq
Both $N$ and $M$ are isomorphic to $\ZZ^n$ but we do not specify
the isomorphism (see, e.g., the proof of Lemma \ref{PIntegralLem}
for this point). The notation
is useful as it serves to distinguish
between objects living in one
lattice and its dual (although of
course in computations
we simply work on $\ZZ^n$, see \S\ref{ExamplesSec}).
Denote the corresponding 
$\RR$-vector space and its dual (both isomorphic to $\RR^n$)
\beq
\lb{NRMREq}
N_\RR:=
N\otimes_\ZZ\RR,
\quad
M_\RR:=M\otimes_\ZZ\RR=N_\RR^*. 
\eeq
A rational convex polyhedral cone in $N_\RR$ takes the form
\beq
\lb{conefirstEq}
\sigma=\sigma(v_1,\ldots,v_d):=
\Big\{
{\textstyle\sum}_{i=1}^da_iv_i\,:\, a_i\ge0, v_i\in N
\Big\}.
\eeq
The rays $\RR_{+}v_i, i\in\{1,\ldots,d\}$ are called the generators
of the cone \cite[p. 9]{Ful}. 
They are (1-dimensional) cones themselves, of course.
Our convention will be that the 
\beq
\lb{conventionprimitiveEq}
v_i, \quad i\in\{1,\ldots,d\}, \q
\hbox{are primitive elements 
of the lattice $N$},
\eeq
which means there is no $m\in \NN\setminus\{1\}$ such that 
$v_i/m\in N$.
A cone is called strongly convex if
$\sigma\cap -\sigma =\{0\}$ \cite[p. 14]{Ful}.
A face of $\sigma$ is any intersection of $\sigma$
with a supporting hyperplane.

\bdefn
\label{FanDefn}
A {\it fan} $\Delta=\{\sigma_i\}_{i=1}^\delta$ in $N$ is a finite set of rational strongly convex polyhedral
cones $\sigma_i$ in $N_\RR$ such that: 

\smallskip
\noindent 
(i) each face of a cone in $\Delta$ is also (a cone) in $\Delta$,

\smallskip
\noindent
(ii) the intersection of two cones in $\Delta$ is a face of each. 
\edefn

Such a fan gives rise to a toric variety $X(\Delta)$: 
each cone $\sigma_i$ in $\Delta$ gives rise to an affine toric variety \cite[\S1.3]{Ful},
that serves as (a Zariski open) chart in $X(\Delta)$ with the transition
between the charts constructed by (i) and (ii) above \cite[p. 21]{Ful}.
For instance, the zero cone corresponds to the open dense orbit
$(\CC^*)^n$ \cite[p. 64]{Cannas}, and more generally there is a
bijection between the cones $\{\sigma_i\}_{i=1}^\delta$ and the orbits
of the complex torus $(\CC^*)^n$ in $X(\Delta)$
\cite[Proposition 5.6.2]{Cannas}, with 
the non-zero cones corresponding precisely to 
all the toric subvarieties of $X(\Delta)$ of positive codimension.

When $X(\Delta)$ is a smooth
toric Fano variety (as we always assume), the fan $\Delta$
must arise from 
an integral polytope as 
follows (but in general, i.e., for 
singular toric varieties, this need not be the case \cite[p. 25]{Ful}).
Let $\Delta$ be a fan such that
$X(\Delta)$ is smooth Fano and let $\sigma_1,\ldots,\sigma_d$ be its 1-dimensional cones
(i.e., rays) generated by primitive generators 
\beq
\lb{Delta1Eq}
\Delta_1:=
\{v_1,\ldots,v_d\}\subset N,
\eeq
so $\sigma_i=\RR_{+}v_i$, and set (recall \eqref{coA})
\beq
\lb{QEq}
Q:=\co\Delta_1=\co\{v_1,\ldots,v_d\}=\co\Delta_1\subset N_\RR.
\eeq
Then $\Delta$ is equal to the 
collection of 
cones over each face of $Q$ plus the zero cone
\cite[p. 26]{Ful}, in other words if $F\subset Q$
is a face, then 
\beq
\lb{sigmaFEq}
\sigma_F:=\{rx\in N_\RR\,:\, r\ge0, x\in F\}
\eeq
is the union of all rays
through $F$ and the origin, and
$$
\Delta=\{\sigma_F\}_{F\subset Q}.
$$
Denote by
\beq
\lb{VerAEq}
\Ver A
\eeq
the vertices of a polytope $A$.
Note that $\Ver F\subset \Ver Q=\Delta_1$, and by \eqref{conefirstEq}--\eqref{conventionprimitiveEq},
\beq
\lb{sigmaFVerFEq}
\sigma_F=\sigma(\Ver F).
\eeq
Smoothness of $X$ means that the generators of $\sigma_F$
form a $\ZZ$-basis for $N$ 
\cite[p. 29]{Ful}. By \eqref{sigmaFVerFEq} this means 
\beq
\lb{smoothnessverticesEq}
\hbox{the
vertices of $F$ form a $\ZZ$-basis for $N$ (for any facet $F\subset Q$).}
\eeq
Thus each facet $F$ of $Q$ is an 
$(n-1)$-simplex whose vertices form a $\ZZ$-basis of $N$.
In Lemma \ref{PIntegralLem} we show this means the vertices of the polar polytope
belong to the dual lattice $M$.

When $L=-K_X$, there is an $\Aut X$ action on $H^0(X,-kK_X)$
for every $k\in\NN$. To get an induced linear action on $M_\QQ$
we must restrict to the normalizer $N((\CC^*)^n)$ of the complex torus $(\CC^*)^n$
in $\Aut X$. The representation of $(\CC^*)^n$ on $H^0(X,-kK_X)$
splits into 1-dimensional spaces, whose generators are called the monomial
basis. There is a one-to-one correspondence between the monomial basis of 
$H^0(X,-kK_X)$ and points in $kP\cap M$, and the  
quotient $N((\CC^*)^n)/(\CC^*)^n$ is a linear group,
that can be identified with $\Aut P\subset
GL(M)\cong GL(n,\ZZ)$
\eqref{AutPEq}. Since $P$ is defined as the convex hull
of vertices in $M$ it follows that $\Aut P$ is finite. 
Alternatively, this can be seen by observing that 
$N_\RR$ is canonically isomorphic to the quotient of  $(\CC^*)^n$ 
by its
maximal compact subgroup $(S^1)^n$ \cite[p. 229]{BS99}
and the induced action on $M_\RR$ is then defined
by transposing via the pairing.
Conversely, all  
compact subgroups of $N((\CC^*)^n)$ that contain $(S^1)^n$ are generated
by $(S^1)^n$ and a finite subgroup $H$ of $\Aut P$
\cite[Proposition 3.1]{BS99}, and we denote
such a group by $G(H)\subset \Aut X$ as in \eqref{GHEq}. We describe in the proof
of Lemma \ref{hkGinvLem} concretely
how the action of $\Aut P$ is expressed in coordinates.
For a finite group $H$ or finite set $\cA$ we denote by 
$$
|H|, \hbox{\ respectively\ } |\cA|,
$$
its order or cardinality.

Oftentimes we will work with 
\beq
\lb{PQcircDef}
P:=-Q^\circ
=
\{-v_1,\ldots,-v_d\}^\circ=
-\{v_1,\ldots,v_d\}^\circ
\subset M_\RR,
\eeq
as it has the nice geometric property of faces of (real) dimension $k$ 
corresponding to toric subvarieties of dimension $k$,
and since the metric properties (e.g., volume) 
of $P$ correspond to those of $X$. (Moreover, $P$ can also be realized as 
the Delzant (moment) polytope associated to any $(S^1)^n$-invariant \K metric
representing the anticanonical class.) 
In particular,
\beq
\lb{PDef2Eq}
P=\bigcap_{i=1}^d\Big\{y\in M_\RR\,:\, \langle y, -v_i\rangle \le 1\Big\}
=\{h_{-\Delta_1}\le 1\}=\{y\in M_\RR\,:\, \max_j\langle -v_j,y\rangle\le1\},
\eeq
and
irreducible toric divisors correspond to facets of $P$
\beq
\lb{toricDiEq}
D_i:=\{y\in P\,:\, \langle y, -v_i\rangle = 1\}.
\eeq
Note that $P$ contains the origin in its interior.
Also note that \eqref{PQcircDef} is the standard convention since
then lattice points of $P$ correspond to monomials via \eqref{skuEq}.
Batyrev--Selivanova use $-P$ instead.

\blem
\lb{PIntegralLem}
Let $P\subset M_\RR$ be the polytope associated to a smooth toric variety $X$. Then $P$ is an integral lattice polytope, i.e., $\Ver P\subset M$.
\elem
\bpf
By duality, if $u\in M_\RR$ is a vertex of $P$ 
then 
\beq
\lb{FQverPEq}
F=\{v\in Q\,:\,\langle u,-v\rangle=1\}\subset N_\RR
\eeq
is a facet of $Q=-P^\circ\subset N_\RR$.
Since $X$ is smooth, the vertices of $F$ form a $\ZZ$-basis for $N$ 
by \eqref{smoothnessverticesEq}.
Choose coordinates on $N$ associated to this $\ZZ$-basis, i.e., the vertices of $F$ are the standard basis vectors $e_1,\ldots,e_n$. 
Thus, the facet $F=\Conv\{e_1,\ldots,e_n\}$ is the standard $(n-1)$-simplex in $\RR^n$
cut-out by the equation $\langle u,-v\rangle=1$ where $u=(-1,\ldots,-1)\in M$,
as desired.

Alternatively, if one does not wish to choose coordinates but rather
work invariantly, denote by $\Ver F=\{f_1,\ldots,f_n\}\subset N$,
and note $\Span_{\ZZ}\Ver F=N$. 
Thus, any 
$v\in N$ can be written uniquely as $v=\sum_{i=1}^na_if_i$ (with $a_i\in\ZZ$),
and so $u\in M_\RR=\mathrm{Hom}(N,\RR)$ can
be identified (recall \eqref{NMHomEq}) with the map
$$
N\ni v\mapsto -\sum_{i=1}^na_i,
$$
As $\sum_{i=1}^na_i\in \ZZ$, this map actually belongs to 
$\mathrm{Hom}(N,\ZZ)=M$.
\epf

Another useful fact is a sort of maximum principle for convex polytopes,
saying essentially that a convex function on a polytope achieves its maximum at some vertex (regardless of continuity).

\blem
\lb{ConvMaxAttaindVertexLem}
Let $A$ be a convex polytope and $f:A\ra\RR\cup\{\infty\}$ a convex function.
Then:

\noindent
(i) $\sup_A f=\sup_{\Ver A} f$. 

\noindent
(ii) if $f$ is bounded on $\Ver A$
it is bounded on $A$ and its maximum is achieved in a vertex.

\noindent
(iii) if $f$ attains its finite maximum on $\Int A$ it is constant.

\noindent
(iv) if $f$ attains its finite maximum on the relative interior of a face
$F\subset A$ it is constant on $F$.
\elem
\bpf
Write $\Ver A=\{p_1,\ldots,p_\ell\}$ and assume $f(p_1)\leq\cdots\leq f(p_\ell)$.
By convexity, $A=\co\Ver A$, and \eqref{coAFiniteSet} implies that 
any $x\in A$ can be expressed as 
\beq
\lb{lambdasimplexEq}
        x=\sum_{i=1}^\ell\lambda_ip_i, \q \sum_{i=1}^\ell\lambda_i=1, 
        \q \lambda\in[0,1]^\ell.
\eeq
Then    $
        f\left(x\right)\leq\sum_{i=1}^\ell\lambda_if\left(p_i\right)\leq\sum_{i=1}^\ell\lambda_i 
        f\left(p_\ell\right)=f\left(p_\ell\right),
    $
    proving (i) and (ii).

    To see (iii), again express any $x\in A$ using \eqref{lambdasimplexEq}.
Suppose that $x_{\Int} \in\Int A$ achieves the (finite) maximum of $f$.
    Note that $x_{\Int}\in\Int A$ means one has the representation \eqref{lambdasimplexEq} for $x_{\Int}$ for some  $\lambda_{\Int}\in(0,1)^\ell$.
    Choose $\delta=\delta(x)\in(0,1)$ so that 
    $\lambda_{\Int}-\delta\lambda\in\RR_+^\ell$.
    Define
    $$
        \lambda':=\frac{1}{1-\delta}\left(\lambda_{\Int}-\delta\lambda\right)\in\RR_+^\ell.
    $$
Note that $\lambda'$ satisfies $\sum_{i=1}^{\ell}{\lambda'}_i=1$, so
letting $x'=\sum_{i=1}^{\ell}{\lambda'}_ip_i$ we have $x'\in A$ by \eqref{lambdasimplexEq}.
    Also,
    $
        \delta x+\left(1-\delta\right)x'=
        {x}_{\Int}.
    $
Thus, 
$
\max_Af= f(x_{\Int}) 
\le
        \delta f(x)+\left(1-\delta\right)f(x')
        \le\max_A f$,
        forcing equality, i.e., $f(x)=\max_Af$ (since $\delta>0$
        and $\max_Af<\infty$), proving (iii).
The proof of (iv) is identical by working on the polytope $F$.
\epf

Finally, we recall Ehrhart's theorem on the polynomiality of the number of lattice points
in dilations of lattice polytopes
\cite{Ehr67a,Ehr67b}, \cite[Theorem 19.1]{Gru07}. Set
    \beq
    \lb{EPkEq}
        E_P(k):=|kP\cap M|, \q k\in\NN.
    \eeq

\bprop
\label{Ehrhart polynomial}
    Let $M$ be a lattice and $P\subset M_\RR$ be a lattice polytope of dimension $n$. Then,
    $$
        E_P(k)=\sum_{i=0}^na_ik^i, \q \hbox{ for any $k\in\NN$},
    $$
    with $a_n=\Vol(P)$, and
\beq
\lb{MonotonicityErhardtEq}
\hbox{$k\leq k'$}\q\Rightarrow\q E_P\left(k\right)\leq E_P\left(k'\right).    
\eeq
\eprop

\section{A natural equivariant Hermitian metric and volume form}
\lb{NaturalvolumeformSec}
In light of Lemma \ref{DefhmuLem} below it makes sense to choose
a convenient pair $(\mu,h)$ of a volume form and a Hermitian metric. 
In fact, we are free to choose such
a pair for each $k$. The special feature of working with $L=-K_X$
is that in fact a volume form essentially doubles as a Hermitian metric,
which is sometimes a bit confusing to keep track of in terms of
notation, but is quite convenient for computations. 
This section serves to explain this choice $(\mu_k=h_k^{1/k},h_k)$,
see \eqref{-kK_X metric-h_k} and \eqref{canonicalomeganEq},
originally due to Song \cite[Lemma 4.3]{Song}.
We emphasize that 
the Hermitian metric must additionally be chosen $G$-invariant in Definition \ref{alphakGhmuDef}, 
and this is confirmed for $h_k$ in Lemma \ref{hkGinvLem}.

Let $X$ be a toric Fano manifold and $P$ its associated polytope. 
There is a natural basis of the space of holomorphic sections $H^0(X,-kK_X)$ defined by the monimials $z^{ku}$ where $u\in P\cap\frac{1}{k} M$. 
That is, there exists an invariant frame $e$ over the open orbit such that 
\beq
\lb{skuEq}
s_{k,u}(z)=z^{ku}e.
\eeq
What does $e$ actually look like?
This is most naturally
expressed in terms of the {\it monomial basis}. When $k=1$ and $u\in P\cap M$
\cite[\S4.3]{Ful},
$$
s_{1,u}=z^u\prod_{i=1}^nz_i\cdot\partial_{z_1}\wedge\cdots\wedge\partial_{z_n}.
$$
In general, for any $k\in\NN$ and $u\in P\cap \frac1k M$,
\begin{equation}\label{-kK_X sections}
    s_{k,u}=z^{ku}\left(\prod_{i=1}^nz_i\right)^k(\partial_{z_1}\wedge\cdots\wedge\partial_{z_n})^{\otimes k}.
\end{equation}
In other words,
\begin{equation}\label{whatiseEq}
    e=\left(\prod_{i=1}^nz_i\right)^k(\partial_{z_1}\wedge\cdots\wedge\partial_{z_n})^{\otimes k}.
\end{equation}

Next, let us construct a canonical Hermitian metric $h_k$ on $-kK_X$. 
Since $-kK_X$ is very ample \cite[p. 70]{Ful}, it is natural to pull-back the Fubini--Study Hermitian metric via the Kodaira embedding. It turns
out that choosing the Kodaira embedding given by the monomial basis 
\begin{equation}\label{iotakEq}
\iota_k:X\ni z\mapsto [s_{k,u}(z)/e(z)]_{u\in P\cap \frac1k M}=[z^{ku}]_{u\in P\cap \frac1k M}\in \PP^{E_P(k)-1},
\end{equation}
will yield the desired $h_k$; importantly, the resulting $h_k$
will be {\it torus-invariant}, smooth, and essentially transform computations
on $X$ to $P$.
To wit, the Fubini--Study metric on $\cO(1)\rightarrow
\PP^{E_P(k)-1}$ is (where $E_P(k)$ is the number of lattice
points in $kP$)
$$
h_{FS}(Z_i,Z_j):=\frac{Z_i\bar Z_j}{\sum_\ell|Z_\ell|^2},
$$
and we define
\beq
\lb{hkiotaEq}
h_k:=\iota_k^*h_{FS}.
\eeq
Note that each homogeneous coordinate $Z_i\in H^0(\PP^{E_P(k)-1},\cO(1))$ pulls-back
via $\iota_k$ to one of the monomial sections $s_{k,u}$ (which one depends 
on the ordering for the elements of $P\cap \frac1k M$ chosen in \eqref{iotakEq}). 
Thus, to express $h_k$ it suffices to compute it on the monomial basis of $H^0(X,-kK_X)$:
\begin{equation}\label{-kK_X metric}
    h_k(s_{k,u_1},s_{k,u_2})(z)
    =
    h_{FS}\big(s_{k,u_1},s_{k,u_2}\big)(\iota_k(z))
    =
    \frac{z^{ku_1}\overline{z^{ku_2}}}{\sum\limits_{u\in P\cap \frac1k M}|z^{ku}|^2}.
\end{equation}
Comparing \eqref{-kK_X sections} and \eqref{-kK_X metric} means that $h_k$ can be written as
\begin{equation}\label{-kK_X metric-h_k}
h_k=\frac{\left(dz^1\wedge \overline{dz^1}\wedge\cdots\wedge
dz^n\wedge\overline{dz^n}\right)^{\otimes k}}{
\left(\prod\limits_{i=1}^n|z_i|^2\right)^k\sum\limits_{u\in P\cap \frac1k M}\left|z^{ku}\right|^2}.
\end{equation}

In conclusion, $h_k^{1/k}$ is a smooth metric on $-K_X$ ($h_k$ being obtained
as a pull-back of a smooth metric under the Kodaira {\it embedding}), 
hence it is a smooth volume form on $X$.
To express this volume form, 
on the open orbit $(\CC^*)^n=\RR^n\times (S^1)^n$ consider the holomorphic coordinates
\beq
\lb{CncoordwEq}
w_i:=x_i/2+\sqrt{-1}\theta_i=\log z_i\in\CC^n.
\eeq
In these coordinates then, this volume form, on the open orbit, is
    \begin{align}
\label{canonicalomeganEq}
\mu_k:=h_k^\frac{1}{k}=\frac{dx_1\wedge\cdots\wedge dx_n\wedge d\theta_1\wedge\cdots\wedge d\theta_n}{\left(\sum\limits_{u\in P\cap M/k}e^{\langle ku,x\rangle}\right)^{\frac{1}{k}}}.
\end{align}
\blem
\lb{hkGinvLem}
Let $G(H)\subseteq\Aut X$ \eqref{GHEq} be a subgroup generated by $(S^1)^n$ and a subgroup $H$ of $\Aut(P)$. Then $h_k$ \eqref{-kK_X metric-h_k} is $G(H)$-invariant.
\elem
\bpf
From \eqref{canonicalomeganEq} it is evident that $h_k$ is independent
of $(\theta_1,\ldots,\theta_n)$, i.e., it is $(S^1)^n$-invariant.
An automorphism $\sigma\in\Aut P\subseteq GL(M)$
can be represented (via choosing a basis for the lattice $M$) 
by a matrix in $GL(n,\ZZ)\cong GL(M)$. Since $\sigma$
preserves the polytope $P$, then $\det\sigma\in\{\pm 1\}$
($\sigma$
could be orientation-reversing, e.g., in the case of a reflection).
The induced action
of $\sigma$ on the dual space $N_\RR$ is naturally represented (via 
the pairing between $M$ and $N$) by the transpose matrix, that we
denote by  $\sigma^T$, and this action is actually coming
from the $\CC$-linear action of $\sigma^T$ on $N_\CC\cong\CC^n$
\eqref{CncoordwEq}.
Thus,
$$
\sigma.(dx_1\wedge\cdots\wedge dx_n)
=
d(x_1\circ\sigma)\wedge\cdots\wedge d(x_n\circ\sigma)
=
\det(\sigma^T)dx_1\wedge\cdots\wedge dx_n
$$
(here $\sigma.$ denotes the action of $\sigma$ on forms, i.e., by pull-back),
and
$\sigma.(d\theta_1\wedge\cdots\wedge d\theta_n)
=
\det(\sigma^T)
d\theta_1\wedge\cdots\wedge d\theta_n
$.
Since $(\det\sigma^T)^2=1$ it remains to consider the denominator
of \eqref{canonicalomeganEq}: 
\begin{align*}
\sigma.
\sum\limits_{u\in P\cap M/k}e^{k\langle u,x\rangle}
&=
\sum\limits_{u\in P\cap M/k}e^{k\langle u,\sigma^T.x\rangle}\cr
&=
\sum\limits_{u\in P\cap M/k}e^{k\langle \sigma.u,x\rangle}\cr
&=
\sum\limits_{u\in \sigma(P\cap M/k)}e^{k\langle u,x\rangle}\cr
&=
\sum\limits_{u\in P\cap M/k}e^{k\langle u,x\rangle},\cr
\end{align*}
since $\sigma$ preserves both $P$ and $M/k$.
In particular, $h_k$ is invariant under $\sigma$, concluding the proof.
\epf
\bremark
An alternative, more invariant, proof of Lemma \ref{hkGinvLem}
is as follows. By \eqref{hkiotaEq}, $\sigma.h_k=(\iota_k\circ\sigma)^*h_{\FS}$.
Now $\iota_k\circ\sigma$ induces the exact same Kodaira embedding 
if $\sigma\in (S^1)^n<G(H)$. So it suffices to consider $\sigma\in H$; then, since $H$ preserves 
$P\cap M/k$, one obtains the same Kodaira embedding up to permutation of the 
coordinates in $\PP^{E_P(k)-1}$. Either way, one obtains the same pull-back
of the Fubini--Study metric as can be from the definition of the Fubini--Study metric
or directly from \eqref{-kK_X metric-h_k}, i.e., $\sigma.h_k=h_k$.
\eremark

\section{An algebraic $\alpha_{k,G}$-invariant and a Demailly
 type result}
\lb{DemaillySec}

\subsection{Analytic definition}

Analogous to the classical $\alpha_G$-invariant, Tian \cite[p. 128]{Tian90} defined the $\alpha_{k,G}$-invariant. 
For this one restricts to a $G$-invariant 
subset of $\cH_k$. To write the subset explicitly in terms of global
\K potentials it is necessary to choose a continuous $G$-invariant 
Hermitian metric $h$ on $-kK_X$:
\begin{equation}
\label{H_k^G definition}
\cH_k^{G}(h):=
\left\{\varphi=\frac{1}{k}\log\sum_i|s_i|_{h}^2\,:\,\varphi\text{ is $G$-invariant, $\{s_i\}$ is a basis of $H^0(X,-kK_X)$}\right\}\subset C^\infty(X).
\end{equation}

\begin{definition}
\lb{alphakGhmuDef}
    Let $G\subseteq\Aut X$ be a compact subgroup. 
    Let $h$ be a fixed continuous $G$-invariant Hermitian metric 
on $-kK_X$, and $\mu$ a fixed continuous volume form on $X$. Then
$$
\alpha_{k,G}(h,\mu):=\sup
\left\{c>0\,:\,\sup_{\varphi\in \cH_k^G}\int_Xe^{-c(\varphi-\sup\varphi)}d\mu<\infty\right\}.
$$
\end{definition}

\blem
\lb{DefhmuLem}
Definition \ref{alphakGhmuDef} does not depend 
on the choice of $h$ or $\mu$.
\elem
For this reason we will simply denote the invariants by $\alpha_{k,G}$
from now on. 

\bpf
Since $X$ is compact, any two continuous volume forms are uniformly bounded
and hence define the same $L^1$ spaces. Next, given two continuous
$G$-invariant Hermitian metrics $h$ and $\tilde h$ on $-kK_X$ 
there is an isomorphism
from $\cH_k^G(h)$ to $\cH_k^G(\tilde h)$ given by 
$
\varphi\mapsto \varphi+\frac1k\log\frac{\tilde h}h.
$
Observe that $\frac1k\log\frac{\tilde h}h$ is (again by compactness of $X$)
a uniformly bounded function on $X$. Hence, for a fixed $c>0$, 
$$
\sup_{\varphi\in \cH_k^G(h)}\int_Xe^{-c(\varphi-\sup\varphi)}d\mu<\infty
\quad \Leftrightarrow 
\sup_{\varphi\in \cH_k^G(\tilde h)}\int_Xe^{-c(\varphi-\sup\varphi)}d\mu<\infty,
$$
as desired.
\epf

\subsection{Algebraic definition}

Consider a (complex) non-zero vector subspace $V$ of $H^0(X,kL)$.
Associated to it is the (not necessarily complete) 
linear system $|V|:=\PP V\subset |kL|:=\PP H^0(X,kL)$ 
\cite[p. 137]{GH}. 
    \blem
    Let $V$ be vector subspace of $H^0(X,kL)$ of dimension $p>0$.
For any basis $\nu_1,\ldots,\nu_p\in H^0(X,kL)$ of $V$, the number
$$
    \sup\left\{c>0\,:\,~ \left(\sum_{j=1}^p|\nu_j(z)|^2\right)^{-{c}} \text{ is locally integrable on }X \right\}
$$
is the same.    \elem
\bpf
Let $\{\nu^{(\ell)}_1,\ldots,\nu^{(i)}_p\}, \, \ell\in\{1,2\},$ be two bases for $V\in H^0(X,kL)$.
Let $A\in GL(p,\CC)$ be the change-of-basis matrix, i.e., 
$\nu^{(2)}_j=A^i_j\nu^{(1)}_i$. Observe that $A^HA$ is a positive
Hermitian matrix-valued on $X$, and denote its eigenvalues 
$0<\lambda_1\le\cdots\le \lambda_p$.
Denote
$\nu^{(\ell)}(z):=(\nu^{(\ell)}_1(z),\ldots,\nu^{(\ell)}_p(z))\in\CC^p$ and
$|\nu^{(\ell)}(z)|^2=\sum_{i=1}^p|\nu^{(\ell)}_i(z)|^2$.
Then,
$$
\frac{|\nu^{(2)}(z)|^2}{|\nu^{(1)}(z)|^2}=
\frac{|A\nu^{(1)}(z)|^2}{|\nu^{(1)}(z)|^2}\in  
[\lambda_1,\lambda_p],
$$
Thus, $|\nu^{(2)}|^2$ is locally integrable
if and only if $|\nu^{(1)}|^2$ is.
\epf
Thus, define the log canonical threshold of the linear system $|V|$ by
\beq
    \lb{lctSigmaEq}
        \lct|V|
        :=\sup\left\{c>0\,:\,~ \left(\sum_j|\nu_j(z)|^2\right)^{-{c}} \text{ is locally integrable on }X \right\}.
    \eeq
   
When $L=-K_X$, there is an $\Aut X$ action on $H^0(X,-kK_X)$
for every $k\in\NN$. Demailly \cite[Theorem A.3, (A.1)]{CS08} noted that then 
$\alpha_G$-invariants (for compact subgroup $G\subset\Aut X$) can be algebraically computed as
    \begin{align}
    \lb{DemaillyA1Eq}
        \alpha_G
        =\inf_{k\in\NN}k\inf_{\substack{|V|\subset|-kK_X|\\ V^G=V\not=0}}\lct
        |V|,
    \end{align}
    where 
    $$
    V^G:=\{v\in V\,:\, g.v\in V\q \forall g\in G\}.
    $$
    Note that in \eqref{DemaillyA1Eq}
    $V\not=0$ ranges over all $G$-invariant vector subspaces of $H^0(X,-kK_X)$ (i.e.,
    of any positive dimension).
    For example, $\lct|-kK_X|=\infty$. 
    From the definition, if $V_1\subset V_2$ are
    two such subspaces it suffices to compute $\lct|V_1|$ since 
    \beq
    \lb{lctinclusionEq}
    \lct|V_1|\le \lct|V_2|.
    \eeq

    It is thus natural to define the algebraic counterpart of the $\alpha_{k,G}$-invariant 
    as follows.
    \begin{definition}
    \lb{glctkDef}
        Let $X$ be a Fano manifold and $G\subset\Aut X$ be a compact subgroup of the automorphism group. Define
    \begin{align}
    \label{alg alpha_k,G}
        \glct_{k,G}
        :=k\inf_{\substack{|V|\subset|-kK_X|\\ V^G=V}}\lct
        |V|.
    \end{align}
    \end{definition}

\subsection{Characterization of the equivariant Bergman spaces}
\lb{BergmanSubSec}

First, we show that in the toric setting the space $\cH_k^{G(H)}$
consists of \K potentials induced by Kodaira embeddings of multiples of 
monomials sections, with the norming constants constant along orbits
of $H$.

Denote by 
\beq
\lb{OrbitsOiEq}
O_1^{(k)},\ldots,O_N^{(k)},
\eeq
the orbits of $H$ in $k^{-1}M\cap P$
(see Figure \ref{P2figure} for an example).

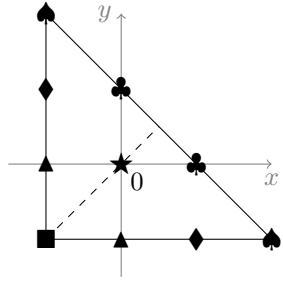
\begin{figure}
        \centering
        \begin{tikzpicture}
            \draw[->,gray](-1.5,0)--(2,0)node[below]{$x$};
            \draw[->,gray](0,-1.5)--(0,2)node[left]{$y$};
            \draw(0,0)node[below right]{$0$};
            \draw(2,-1)node{$\spadesuit$}--(1,0)node{$\clubsuit$}--(0,1)node{$\clubsuit$}--(-1,2)node{$\spadesuit$}--(-1,1)node{$\blacklozenge$}--(-1,0)node{$\blacktriangle$}--(-1,-1)node{$\blacksquare$}--(0,-1)node{$\blacktriangle$}--(1,-1)node{$\blacklozenge$}--cycle;
            \draw(0,0)node{$\bigstar$};
            \draw[dashed](-1,-1)--(1/2,1/2);
        \end{tikzpicture}
        \caption{The six orbits $O_1^{(1)},\ldots,O_6^{(1)}$ of the action of the group generated by the reflection about $y=x$ on the polytope corresponding to $\PP^2$ with $k=1$.}\lb{P2figure}
    \end{figure}

\blem
\lb{HkGHtoricLem}
Let $\{s_{k,u}\}_{u\in k^{-1}M\cap P}$ be
the monomial basis \eqref{skuEq} of $H^0(X,-kL)$.
Let $G(H)\subseteq\Aut X$ \eqref{GHEq} be a subgroup generated by $(S^1)^n$ and a subgroup $H$ of $\Aut(P)$. 
Then (recall (\ref{H_k^G definition}) and (\ref{OrbitsOiEq})),
$$
\cH_k^{G(H)}:=
\cH_k^{G(H)}(h_k)=
\left\{{k}^{-1}\log\sum_{i=1}^N\lambda_i\sum_{u\in O_i^{(k)}}|s_{k,u}|_{h_k}^2\,
:\, \lambda_i>0\right\}.
$$
\elem

\bpf
Recall the natural basis $\{s_{k,u}\}_{u\in P\cap M/k}$ defined by the
lattice monomials \eqref{skuEq}. Given any other basis $\{s_i\}$ for $H^0(X,-kK_X)$, let 
$A\in GL(E_P(k),\CC)$ be the change-of-basis matrix, so 
$$
s_i=A^u_is_{k,u}, \q i\in\{1,\ldots,E_P(k)\}
$$
(we use the Einstein summation convention).

Let $\vp\in\cH_k^{G(H)}$ and let $\{s_i\}$ be
the associated basis \eqref{H_k^G definition}.
Expanding $\vp$ in terms of the monomials,
\begin{align*}
e^{k\varphi}
&=\sum_i|s_i|_{h_k}^2=\sum_{u,u'\in P\cap M/k}
\sum_{i,j=1}^{E_P(k)}A^u_i\overline{A_j^{u'}}
\langle s_{k,u},s_{k,u'}\rangle_{h_k}
\cr
&=:
\left|s_{k,0}\right|_{h_k}^2\sum_{u,u'\in P\cap M/k}c_{u,u'}z^{ku}\bar{z}^{ku'}
\end{align*}
for some coefficients $\{c_{u,u'}\}_{u,u'\in P\cap M/k}$. 
We claim that $\{z^u\bar{z}^{u'}\}_{u,u'\in M}$ are linearly independent. 
To see that, suppose
$$
f(z)=\sum_{u,u'\in M}c_{u,u'}z^u\bar{z}^{u'}=0
$$
with all but finitely many coefficients being 0. We may assume for any $c_{u,u'}\neq0$, $u,u'$ lie in the positive orthant, by multiplying by $\displaystyle\prod_{i}\left|z_i\right|^2$ to some sufficiently large power. Then
$$
c_{u,u'}=\frac{1}{u!u'!}\left.\frac{\partial^{\left|u\right|+\left|u'\right|}}{\partial^uz\partial^{u'}\bar{z}}\right|_{z=0}f(z)=0,
$$
proving the claim. 
Now, the subgroup $(S^1)^n<G(H)$ acts on the open orbit
$(\CC^*)^n$ by 
\beq
\lb{torusactionEq}
(\beta_1,\ldots,\beta_n).(z_1,\ldots,z_n)=
(e^{\i\beta_1}z_1,\ldots,e^{\i\beta_n}z_n).
\eeq
So by $(S^1)^n$-invariance and Lemma \ref{hkGinvLem},
$$
\beta.e^{k\varphi}=
\left|s_{k,0}\right|_{h_k}^2\sum_{u,u'\in P\cap M/k}c_{u,u'}e^{\i\left\lan\beta,k\left(u-u'\right)\right\ran}z^{ku}\bar{z}^{ku'}
=
\left|s_{k,0}\right|_{h_k}^2\sum_{u,u'\in P\cap M/k}c_{u,u'}z^{ku}\bar{z}^{ku'}
=e^{k\varphi},
$$
for any $\beta\in(S^1)^n=(\RR/2\pi\ZZ)^n$.
By comparing the coefficients and using the claim we just demonstrated, it follows that for $e^{k\varphi}$ to be $(S^1)^n$-invariant we must have $c_{u,u'}=0$ whenever $u\neq u'$
(the converse is also true, of course).

Moreover, we also require $e^{k\varphi}$ to be invariant under the action of $H$, i.e., $c_{u,u}=c_{\sigma u,\sigma u}$ for any $\sigma\in H$. In conclusion, let $O^{(k)}_1,\ldots,O^{(k)}_N$ be the orbits of the action of $H$ on $P\cap M/k$. Then
$$
e^{k\varphi}=\sum_{i=1}^N\lambda_i\sum_{u\in O^{(k)}_i}|s_{k,u}|_{h_k}^2,
$$
for some coefficients $\{\lambda_i\}_{i=1}^N$. Thus 
\eqref{H_k^G definition} simplifies to
$$
\cH_k^{G(H)}(h)=\left\{\varphi=\frac{1}{k}\log\sum_{i=1}^N\lambda_i\sum_{u\in O^{(k)}_i}|s_{k,u}|_{h_k}^2\,:\,\lambda_i>0\right\},
$$
as claimed.
\epf

\subsection{Replacing a basis of sections by an orbit of a section}
\lb{DemaillyProbSubSec}

The next result generalizes \cite[Proposition 2.1]{Shi10}
to the equivariant setting using Lemma \ref{HkGHtoricLem}.

\begin{proposition}
\label{analytic alpha_k,G}
For a subgroup $G(H)\subseteq\Aut X$ \eqref{GHEq} 
generated by $(S^1)^n$ and a subgroup $H$ of $\Aut(P)$,
\begin{equation}\label{alpha_{k,G}}
    \alpha_{k,G(H)}=\sup\left\{c>0\,:\,\int_X\Big(\sum_{\sigma\in H}\left|s_{k,\sigma u}\right|_{h_k}^2\Big)^{-\frac{c}{k}}d\mu<\infty,~\forall~u\in P\cap M/k\right\}.
\end{equation}
\end{proposition}
\begin{proof}

\noindent {\it Step 1: estimate a basis-type element using the worst orbit present in its expansion.}
Assuming $\lambda_N=\max_{i=1,\ldots,N}\lambda_i$, we have$$
\sup\varphi\leq\sup\left(\frac{1}{k}\log\sum_{i=1}^N\lambda_N\sum_{u\in O^{(k)}_i}|s_{k,u}|_{h_k}^2\right)=\frac{1}{k}\log\lambda_N,
$$
since 
$
1/h_k=\sum_{u\in P\cap M/k}|s_{k,u}|^2.
$
If $c>0$ is such that for each $i\in\{1,\ldots,N\}$,
\begin{align*}
    \int_X\left(\sum_{u\in O^{(k)}_i}\left|s_{k,u}\right|_{h_k}^2\right)^{-\frac{c}{k}}d\mu\leq C_c,
\end{align*}
then for $\varphi\in\cH_k^{G(H)}$ (using Lemma \ref{HkGHtoricLem}),
\begin{align*}
\int_X e^{-c(\varphi-\sup\varphi)}d\mu
&=
    e^{c\sup\varphi}\int_X\left( \sum_{i=1}^N\lambda_i\sum_{u\in O^{(k)}_i}|s_{k,u}|_{h_k}^2\right)^{-\frac{c}{k}}d\mu\\
    &\leq e^{\frac{c}{k}\log\lambda_N}
    \int_X\left(\lambda_N\sum_{u\in O^{(k)}_N}\left|s_{k,u}\right|_{h_k}^2\right)^{-\frac{c}{k}}d\mu
    \leq C_c.
\end{align*}

\noindent {\it Step 2: estimate an orbit-type element using an approximation
by degenerating basis-type elements.
}

Conversely, assume $\alpha>0$ is such that
for any $\varphi\in \cH_k^{G(H)}$,
\begin{align*}
    \int_X e^{-c(\varphi-\sup\varphi)}d\mu\leq C_c.
\end{align*}
Since, for any $\ell=1,\ldots,N$, 
\begin{align*}
    \left(\sum_{u\in O^{(k)}_\ell}\left|s_{k,u}\right|_{h_k}^2\right)^{-\frac{c}{k}}
    =\lim_{\lambda_i\rightarrow0,i\neq\ell}\left( \sum_{i=1}^N\lambda_i\sum_{u\in O^{(k)}_i}|s_{k,u}|_{h_k}^2\right)^{-\frac{c}{k}},
\end{align*}
where $\lambda_\ell=1$, by Fatou's Lemma \cite[Lemma 2.18]{Folland},
\begin{align*}
    \int_X\left(\sum_{u\in O^{(k)}_\ell}\left|s_{k,u}\right|_{h^k}^2\right)^{-\frac{c}{k}}d\mu
    &\leq \liminf_{\lambda_i\rightarrow0,i\neq\ell} 
    e^{-\alpha\sup\varphi_\lambda}e^{\alpha\sup\varphi_\lambda} 
    \int_X \left( \sum_{i=1}^N\lambda_i\sum_{u\in O^{(k)}_i}|s_{k,u}|_{h_k}^2\right)^{-\frac{\alpha}{k}}d\mu  \\
    &=\left(\sup\sum_{u\in O^{(k)}_\ell}|s_{k,u}|_{h_k}^2\right)^{-\frac{c}{k}}
    \liminf_{\lambda_i\rightarrow0,i\neq\ell} \int_X e^{-c(\varphi_\lambda-\sup\varphi_\lambda)}d\mu\\
    &\leq\left(\sup\sum_{u\in O^{(k)}_\ell}|s_{k,u}|_{h_k}^2\right)^{-\frac{c}{k}}C_c,
\end{align*}
where
\begin{align*}
    \varphi_\lambda
    =\frac{1}{k}\log\sum_{i=1}^N\lambda_i\sum_{u\in O^{(k)}_i}|s_{k,u}|_{h_k}^2.
\end{align*}
Thus, we have shown that 
\begin{align}
    \alpha_{k,G(H)}&=\sup\left\{c>0\,:\,\int_X\left(\sum_{u\in O^{(k)}_i}\left|s_{k,u}\right|_{h_k}^2\right)^{-\frac{c}{k}}d\mu<\infty,~\forall~i\right\}\label{alpha orbit exp}\\
    &=\sup\left\{c>0\,:\,\int_X\left(\sum_{\sigma\in H}\left|s_{k,\sigma u}\right|_{h_k}^2\right)^{-\frac{c}{k}}d\mu<\infty,~\forall~u\in P\cap M/k\right\},\nonumber
\end{align}
proving \eqref{alpha_{k,G}}.
\end{proof}

\begin{corollary}\label{alpha orbit exp cor}
    Fix $k\in\NN$ and let $O^{(k)}_1,\ldots,O^{(k)}_N$ be the orbits of the action of $H$ on $k^{-1}M\cap P$. Then
    $$
        \alpha_{k,G(H)}=\min_{1\leq i\leq N}\sup\left\{c>0\,:\,\int_X\left(\sum_{u\in O^{(k)}_i}\left|s_{k,u}\right|_{h_k}^2\right)^{-\frac{c}{k}}d\mu<\infty\right\}.
    $$
\end{corollary}
\begin{proof}
    By \eqref{alpha orbit exp},
    \begin{align*}
        \alpha_{k,G(H)}&=\sup\left\{c>0\,:\,\int_X\left(\sum_{u\in O^{(k)}_i}\left|s_{k,u}\right|_{h_k}^2\right)^{-\frac{c}{k}}d\mu<\infty,~\forall~i\right\}\\
        &=\min_{1\leq i\leq N}\sup\left\{c>0\,:\,\int_X\left(\sum_{u\in O^{(k)}_i}\left|s_{k,u}\right|_{h_k}^2\right)^{-\frac{c}{k}}d\mu<\infty\right\}.
    \end{align*}
\end{proof}

We can now answer affirmatively Problem \ref{DemaillyProb} in our setting.

\bpf[Proof of Proposition \ref{DemaillyProblemProp}]
    Let $\Big|V_{O^{(k)}_i}\Big|$ denote the linear system generated by 
    $\{(s_{k,u}) 
    \,:\, ~u\in O^{(k)}_i\}$ (recall \eqref{OrbitsOiEq}).
    By \eqref{lctSigmaEq} and \eqref{alpha_{k,G}},
    $$
    \alpha_{k,G(H)}=k\inf_{i}\lct\Big|V_{O^{(k)}_i}\Big|.
    $$
     It remains to show
    $$
        \glct_{k,G(H)}
        =k\inf_{i}\lct\Big|V_{O^{(k)}_i}\Big|.
    $$
     By \eqref{lctinclusionEq}, it suffices
    to restrict to irreducible $G(H)$-invariant linear systems
    (cf. \cite[p. 146]{LSY15}). Now, any $(S^1)^n$-invariant linear system $|V|$ is spanned by monomials (see Lemma \ref{onedimpiecesLema} below). By irreducibility, this means $|V|$ is spanned by the monomials coming from some $H$-orbit $O_i$. i.e., $|V|=\Big|V_{O^{(k)}_i}\Big|$.
\end{proof}

\blem
\lb{onedimpiecesLema}
Let
$V\subseteq H^0(X,-kK_X)$ be a complex vector subspace invariant under $(S^1)^n$. Then
there exists a unique subset $\cF\subset P\cap M/k$ such that
    $$
        V=\Span\left\{s_{k,u}\right\}_{u\in\cF}.
    $$
\elem

\bpf
It suffices to prove that any section in $V$ is generated by monomials in $V$. That is, given any section
$$
s=\sum_{u\in P\cap M/k}a_us_{k,u}\in V,
$$
if $a_u\neq0$, then $s_{k,u}\in V$.

By $(S^1)^n$-invariance, for any $\beta\in(S^1)^n$ and $s\in V$
also $\beta.s\in V$ (recall \eqref{torusactionEq}), i.e., 
$$
\beta.s=
\sum_{u\in P\cap M/k}a_u\beta.s_{k,u}
=
\sum_{u\in P\cap M/k}a_ue^{\i\left\langle\beta,ku 
\right\rangle}s_{k,u}\in V.
$$
Thus for any $u_0\in P\cap M/k$,
\begin{align*}
    \int_{\left(S^1\right)^n}e^{-\i\left\langle\beta,ku_0\right\rangle}\beta.s\,d\beta&=\sum_{u\in P\cap M/k}a_us_{k,u}\int_{\left(S^1\right)^n}e^{\i\left\langle\beta,ku-ku_0\right\rangle}d\beta\\
    &=\sum_{u\in P\cap M/k}a_us_{k,u}\cdot\left(2\pi\right)^n\delta_{u-u_0}\\
    &=\left(2\pi\right)^na_{u_0}s_{k,u_0}\in V.
\end{align*}
If $a_{u_0}\neq0$, then $s_{k,u_0}\in V$. This completes the proof.
\epf

\begin{corollary}\label{alpha_G corollary}
    \begin{equation}\label{alpha_G}
        \alpha_{G(H)}=\inf_k\alpha_{k,G(H)}=\sup\left\{c>0\,:\,\int_X\left(\sum_{\sigma\in H}\left|s_{k,\sigma u}\right|_{h_k}^2\right)^{-\frac{c}{k}}d\mu<\infty,~\forall~u\in P\cap M/k,~k\in\NN\right\}.
    \end{equation}
\end{corollary}
\begin{proof}
    By Demailly's theorem \eqref{DemaillyA1Eq} \cite[(A.1)]{CS08}, \eqref{alg alpha_k,G}, and Proposition \ref{DemaillyProblemProp},
    $$
    \alpha_{G(H)}=\inf_{k\in\NN}\glct_{k,G(H)}=\inf_{k\in\NN}\alpha_{k,G(H)}.
    $$
    By \eqref{alpha_{k,G}},
    \begin{align*}
        \alpha_{G(H)}&=\inf_{k\in\NN}\sup\left\{c>0\,:\,\int_X\left(\sum_{\sigma\in H}\left|s_{k,\sigma u}\right|_{h_k}^2\right)^{-\frac{c}{k}}d\mu<\infty,~\forall~u\in P\cap M/k\right\}\\
        &=\sup\left\{c>0\,:\,\int_X\left(\sum_{\sigma\in H}\left|s_{k,\sigma u}\right|_{h_k}^2\right)^{-\frac{c}{k}}d\mu<\infty,~\forall~u\in P\cap M/k,~k\in\NN\right\}.
    \end{align*}
\end{proof}

\bremark
Note that here it would not have been enough to invoke Tian's theorem
\cite[Proposition 6.1]{Tian90}, and it was necessary to invoke
Demailly's theorem \cite[(A.1)]{CS08}. The reason is that the $\alpha_{k,G}$-invariants
are defined by requiring certain integrals to be merely finite, i.e., bounded, but with a constant possibly depending on $k$. Tian's theorem says that $\cH^G$ is approximated
by $\cH_k^G$, but in approximating $\vp\in\cH^G$ by a sequence $\vp_k\in\cH_k^G$
it might happen that the integrals $\int_Xe^{-c(\vp_k-\sup\vp_k)}\omega^n$
blow up as $k$ tends to infinity. Demailly precludes that from happening
by using the Demailly--Koll\'ar lower semi-continuity of complex singularity exponents.
\eremark

\section{The case of divisible $k$}
\lb{DivisibleSec}

We ultimately improve on the results of this section, but we include
them since they serve to emphasize the difficulties that still need
to be dealt with (see Remark \ref{DivisibleRemark}).  

\begin{proposition}
\lb{weakmonotoneProp}
    Let $G(H)\subseteq\Aut X$ \eqref{GHEq} be a subgroup generated by $(S^1)^n$ and a subgroup $H$ of $\Aut(P)$. For $k,\ell\in\NN$, $\alpha_{k,G(H)}\geq\alpha_{k\ell,G(H)}$.
\end{proposition}
\begin{proof}
    Since $h_{k\ell}$ and $h_k^\ell$ \eqref{-kK_X metric} are both smooth metrics on $-k\ell K_X$, by compactness of $X$, they are equivalent. We also use the following lemma.
    \begin{lemma}
\lb{xlLemma}
        Let $a_1,\ldots,a_N\geq0$. Then for $\ell\in\NN$,
        $$
            \sum_{i=1}^Na_i^\ell\leq\left(\sum_{i=1}^Na_i\right)^\ell\leq N^{\ell-1}\sum_{i=1}^Na_i^\ell.
        $$
    \end{lemma}
    \begin{proof}
        The first inequality follows by expanding $(\sum_{i=1}^Na_i)^\ell$. For the second inequality, consider the function $f(x):=x^\ell$ on $[0,+\infty)$. Since $f$ is convex, by Jensen's inequality,
        $$
            f\left(\frac{1}{N}\sum_{i=1}^Na_i\right)\leq\frac{1}{N}\sum_{i=1}^Nf\left(a_i\right),
        $$
        i.e.,
        $
        \disp    \left(\sum_{i=1}^Na_i\right)^\ell\leq N^{\ell-1}\sum_{i=1}^Na_i^\ell.
        $
    \end{proof}
    By Proposition \ref{analytic alpha_k,G}, and since $M/k\ell \subset M/k$,
    \begin{align*}
        \alpha_{k\ell,G(H)}&=\sup\left\{c>0\,:\,\int_X\Big(\sum_{\sigma\in H}\left|s_{k\ell,\sigma u}\right|_{h_{k\ell}}^2\Big)^{-\frac{c}{k\ell}}d\mu<\infty,~\forall~u\in P\cap M/k\ell\right\}\\
        &\leq\sup\left\{c>0\,:\,\int_X\Big(\sum_{\sigma\in H}\left|s_{k\ell,\sigma u}\right|_{h_{k\ell}}^2\Big)^{-\frac{c}{k\ell}}d\mu<\infty,~\forall~u\in P\cap M/k\right\}\\
        &=\sup\left\{c>0\,:\,\int_X\Big(\sum_{\sigma\in H}\left|s_{k,\sigma u}^{\otimes\ell}\right|_{h_k^\ell}^2\Big)^{-\frac{c}{k\ell}}d\mu<\infty,~\forall~u\in P\cap M/k\right\}\\
        &=\sup\left\{c>0\,:\,\int_X\Big(\sum_{\sigma\in H}\left|s_{k,\sigma u}\right|_{h_k}^{2\ell}\Big)^{-\frac{c}{k\ell}}d\mu<\infty,~\forall~u\in P\cap M/k\right\}\\
        &=\sup\left\{c>0\,:\,\int_X\Big(\sum_{\sigma\in H}\left|s_{k,\sigma u}\right|_{h_k}^2\Big)^{-\frac{c}{k}}d\mu<\infty,~\forall~u\in P\cap M/k\right\}\\
        &=\alpha_{k,G(H)},
    \end{align*}
    where Lemma \ref{xlLemma} was invoked in the penultimate equality.
\end{proof}

\bprop
\lb{divisibleProp}
Let $G(H)\subseteq\Aut X$ \eqref{GHEq} be a subgroup generated by $(S^1)^n$ and a subgroup $H$ of $\Aut(P)$.
There exists $k_0\in\NN$ such that 
$\alpha_{K,G(H)}=\alpha_{G(H)}$ for all $K$ divisible by $k_0$. 
\eprop

\bremark
\lb{Minu0Rem}
In fact, the proof will show that $k_0$ is determined by the fan as follows: 
let $u_0\in P^H$ attain
$\sup_{u\in P^H}\max_i\langle u,v_i\rangle$.
Since $P^H$ is a convex polytope $u_0$ will be a vertex of $P^H$, i.e., cut out by $P^H$ and the supporting hyperplanes of $P$ containing $u_0$. The equations defining $P^H$ are determined by $H\subset GL(M)$ hence are linear equations with integer coefficients. So are the equations cutting out $\partial P$. It follows that $u_0$ is a rational point, i.e., $u_0\in M/k_0$ for some $k_0$. Let $k_0$ be the smallest such positive integer.
The fact that $u_0$ is a rational point is originally due
to Song \cite[p. 1257]{Song} who proved that 
$\alpha_{k_0,G(\Aut P)}=\alpha_{G(\Aut P)}$.

\eremark

\bpf
    We can further simplify \eqref{alpha_G}. 
    Recall \eqref{definition of p}.
    By the geometric-arithmetic mean inequality,
    \begin{align*}
        \left(\frac{1}{|H|}\sum_{\sigma\in H}\left|s_{k,\sigma u}\right|_{h_k}^2\right)^{-\frac{\alpha}{k}}\leq\prod_{\sigma\in H}\left|s_{k,\sigma u}\right|_{h_k}^{-\frac{2\alpha}{|H|k}}
        &=\left|s_{|H|k,\pi_H(u)
        }\right|_{h_{|H|k}}^{-\frac{2\alpha}{|H|k}}
        \cr
        &=\left(\frac{1}{|H|}\sum_{\sigma\in H}\left|s_{|H|k,\sigma\pi_H(u)
        }\right|_{h_{|H|k}}^2\right)^{-\frac{\alpha}{|H|k}},
    \end{align*} 
    with the last equality since $\pi_H(u)$ is fixed by $H$
    (recall \eqref{PHEq}--\eqref{definition of p})
    so each term in the sum is identical.
    Therefore the supremum in \eqref{alpha_G} is unchanged 
    if restricted to those $u$ fixed by $H$. Thus, using
    Lemma \ref{AuxOneSecLCTLemma} (proven below), \eqref{alpha_G} simplifies to
    (recall \eqref{PHEq})
    \begin{align*}
        \alpha_{G(H)}
        &=
        \sup\left\{c>0\,:\,\int_X\left|s_{k,u}\right|_{h_k}^{-\frac{2c}{k}}d\mu<\infty,~\forall~u\in P^H\cap\frac{1}{k}M,~k\in\NN\right\}
        \\
        &=
        \inf\left\{k\cdot\lct(s_{k,u})\,:\,~u\in P^H\cap\frac{1}{k}M,~k\in\NN\right\}
        \\
        &=\inf_{k\in\NN}\inf_{u\in P^H\cap\frac{1}{k}M}\frac{1}{
        \max_i\langle u,v_i\rangle+1}\\
        &=\inf_{u\in P^H}\frac{1}{\max_i\langle u,v_i\rangle+1}\\
        &=\frac1{\max_i\langle u_0,v_i\rangle+1},
    \end{align*}
    where we used the notation of Remark \ref{Minu0Rem}. By that same Remark,
    $u_0\in P^H\cap\frac{1}{K}M$ whenever $K\in\NN$ is divisible by $k_0$. In particular, by \eqref{alpha_{k,G}}, and using Lemma \ref{AuxOneSecLCTLemma} again,
    \begin{align*}
        \alpha_{K,G(H)}&\leq\sup\left\{c>0\,:\,\int_X\left|s_{K,u_0}\right|_{h_K}^{-\frac{2c}{K}}d\mu<\infty\right\}\\
        &=K\cdot\lct(s_{K,u_0})\\
        &=\min_i\frac{1}{\langle u_0,v_i\rangle+1}\\
        &=\alpha_{G(H)}.\\
\end{align*}
Since by definition $\alpha_{G(H)}\le \alpha_{K,G(H)}$, 
    equality is achieved and $\alpha_{K,G(H)}=\alpha_{G(H)}$ for all $K$ divisible by $k_0$. \epf

\blem
\lb{AuxOneSecLCTLemma}
For $u\in P\cap k^{-1}M$,
$\displaystyle\lct (s_{k,u})=\frac1k\frac1{1+\max_i\lan u,v_i\ran}$
(recall (\ref{Delta1Eq}) and (\ref{skuEq})).
\elem
\bpf
This is well-known (see, e.g., \cite[Corollary 7.4]{BJ20}, \cite[Theorem 4.1]{LSY15}).
It is also a consequence of Proposition \ref{Orbit lct} proven
below (put $\cF=\{u\}$).
\epf

\bremark
\lb{DivisibleRemark}
According to Proposition \ref{divisibleProp}, the sequence $\{\alpha_{k,{G(H)}}\}_{k\in\NN}$ 
is constant (equal to $\alpha_{G(H)}$) along the subsequence $k_0,2k_0,\ldots$. 
Proposition \ref{divisibleProp} does not yield information for all $k\in\NN$ unless $k_0=1$,
    and examples show (see \S\ref{ExamplesSec}) 
    that oftentimes $k_0>1$. 
Moreover, even though the sequence also satisfies
$\alpha_{k,{G(H)}}\ge \alpha_{k\ell,{G(H)}}\ge \alpha_{G(H)}$ for any $k,\ell\in\NN$
(Proposition \ref{weakmonotoneProp}), these facts combined are
still not enough to conclude Tian's conjecture without further work. 
For instance, the sequence
$$
a_k:=\begin{cases} \alpha_{G(H)}, & \hbox{\ $k$ even,}\cr
\alpha_{G(H)}+k^{-1}, & \hbox{\ $k$ odd,}\cr
\end{cases}
$$
satisfies these requirements for $k_0=2$ (and also satisfy $\lim_ka_k=\alpha_{G(H)}$). 
Perhaps a more natural sequence that satisfies all the requirements
and even for the $k_0$ of Remark \ref{Minu0Rem} is
$$
a_k:=
\inf\left\{k\cdot\lct(s_{k,u})\,:\,~u\in P^H\cap\frac{1}{k}M\right\},
$$      
with the proof of Proposition \ref{divisibleProp} showing that
$a_{k_0\ell}=a_{k_0}=\alpha_{G(H)}$ for $\ell\in\NN$ but possibly $a_k>\alpha_{G(H)}$
for $k$ not divisible by $k_0$.
    Thus, we are led  to develop more refined
    estimates that are the topic of the next section.
\eremark

\section{Estimating singularities associated to orbits}
\lb{OrbitsEstimateSec}

\subsection{Real singularity exponents and support functions}

In this subsection we develop a key new technical estimate 
that expresses real singularity exponents associated to 
collections of toric monomials in terms of support functions.

\bdefn
    For non-empty finite sets $\cF,\cU\subseteq\RR^n$ and $k\in\NN$,
    $$
        c_k({\cF,\cU}):=\sup\left\{c\in(0,1)\,:\,\int_{\RR^n}\frac{\left(\sum_{u\in \cF}e^{\langle ku,x\rangle}\right)^{-\frac{c}{k}}}{\left(\sum_{u\in \cU}e^{\langle ku,x\rangle}\right)^\frac{1-c}{k}}dx<\infty\right\}.
    $$
\edefn
 
\begin{prop}
\label{generalizeSong}
     For $\cF,\cU\subseteq\RR^n$  non-empty finite sets
    with $0\in\Int \co\cU$ (recall (\ref{coA})), 
\begin{align*}
    c_k({\cF,\cU})
    &=\sup\big\{c\in(0,1)\,:\,0\in (1-c)\co\cU+{c}\co\cF\big\}\cr
    &=\sup\left\{c\in(0,1)\,:\,\Big(-\frac{c}{1-c}\co\cF\Big)\cap 
        \co\cU\neq\emptyset\right\}>0.
\end{align*}
    In particular, $c_k({\cF,\cU})$ is independent of $k$.
\end{prop}

\begin{proof}
Using polar coordinates $x=r\nu$ and $dx=r^{n-1}dr\wedge dS(\nu)$
with $r\in\RR_+,\, \nu\in S^{n-1}(1):=\{x\in\RR^n\,:\, |x|=1\}$,
    \beq
    \lb{integrationpolarcoordEq}
        \int_{\RR^n}\frac{\left(\sum_{u\in \cF}e^{\langle ku,x\rangle}\right)^{-\frac{c}{k}}}{\left(\sum_{u\in \cU}e^{\langle ku,x\rangle}\right)^\frac{1-c}{k}}dx=\int_{S^{n-1}(1)}f_c(\nu)dS(\nu),
    \eeq
where $f_c:S^{n-1}(1)\ra\RR_+$ is defined by
    $$
        f_c(\nu):=\int_0^\infty\frac{\left(\sum_{u\in \cF}e^{kr\langle u,\nu\rangle}\right)^{-\frac{c}{k}}}{\left(\sum_{u\in \cU}e^{kr\langle u,\nu\rangle}\right)^\frac{1-c}{k}}r^{n-1}dr.
    $$
    Notice that for any finite set $\cA\subset\RR^n$ and $r\in\RR_+$
    (recall (\ref{hKEq})),
    $$
        e^{kr h_{\cA}(\nu)}
        \leq
        \sum_{u\in \cA}e^{kr\langle u,\nu\rangle}\leq|\cA|e^{kr h_{\cA}(\nu)}.
    $$
    Hence for $c\in(0,1)$,
    \beq
    \lb{twosidedestimatecUcF}
        \frac{|\cF|^{-\frac{c}{k}}}{|\cU|^\frac{1-c}{k}}e^{-krg_c(\nu)}\leq\frac{\left(\sum_{u\in \cF}e^{kr\langle u,\nu\rangle}\right)^{-\frac{c}{k}}}{\left(\sum_{u\in \cU}e^{kr\langle u,\nu\rangle}\right)^\frac{1-c}{k}}\leq e^{-krg_c(\nu)},
    \eeq
    where
    \begin{align}
    \lb{gcEq}
        g_c(x)&:=c\max_{u\in \cF}\langle u,x\rangle+(1-c)\max_{u\in \cU}\langle u,x\rangle
        \cr
        &=h_{c\cF+(1-c)\cU}(x)\cr
        &=c\max_{u\in \co\cF}\langle u,x\rangle+(1-c)\max_{u\in \co\cU}\langle u,x\rangle
        \cr
        &=ch_{\co\cF}(x)+(1-c)h_{\co\cU}(x)
        =h_{c\co\cF+(1-c)\co\cU}(x),
    \end{align}
    where $A+B:=\{x+y\,:\,x\in A,y\in B\}$ is the Minkowski sum.

We need the following property of support functions.
\bclaim
\lb{supportfnClaim}
Let $\cA\subset\RR^n$ be a nonempty finite set. Then
$h_\cA|_{S^{n-1}(1)}> 0$ if  and only if
$0\in\Int\co\cA$.

\eclaim

\bpf
Note first that 
\beq
\lb{hcAcocAEq}
h_\cA=h_{\co\cA}:
\eeq
$h_\cA\le h_{\co\cA}$ since $\cA\subset\co\cA$, 
while if $h_{\co\cA}(y)=\lan a(y),y\ran$ for some $a(y)\in\co\cA$,
and writing $a(y)=\sum_{i=1}^{|\cA|}\lambda_ia_i$
with $\sum_{i=1}^{|\cA|}\lambda_i=1$ and $\lambda_i\ge0$, 
where $\cA=\{a_i\}$ yields $h_{\co\cA}(y)=\sum_{i=1}^{|\cA|}\lambda_i\lan a_i,y\ran
\le\sum_{i=1}^{|\cA|}\lambda_ih_{\cA}(y)=h_\cA(y)$.

Next, if $0\in\Int\co\cA$ then $\eps B_2^n\subset\co\cA$ where 
$B_2^n:=\{x\in\RR^n\,:\, |x|\le 1\}$ and
$h_{\co\cA}\ge h_{\eps B_2^n}=\eps h_{B_2^n}=\eps$ when restricted
to $S^{n-1}(1)$.

Conversely, if $0\not\in\Int\co\cA$
then by convexity of $\co\cA$ there exists
a hyperplane $H$ passing through a boundary point of $\co\cA$ 
so that $\co\cA$ lies on one
side of it and $0$ lies on the other side of it.
If $\nu\in S^{n-1}(1)$ is normal to $H$ and points toward the side not containing $\co\cA$
then $h_{\co\cA}(\nu)\le0$.
\epf

\bclaim
\lb{supportfn2Claim}
Let $\cA\subset\RR^n$ be a nonempty finite set. Then
$h_\cA|_{S^{n-1}(1)}$ is somewhere negative  if and only if
$0\in \Int(\RR^n\setminus\co\cA)=\RR^n\setminus\co\cA$.

\eclaim
\bpf
Suppose $0\not\in \Int(\RR^n\setminus\co\cA)=\RR^n\setminus\co\cA$.
Equivalently $0\in\Conv\cA$. Then $h_\cA\geq h_{\{0\}}=0$ so $h_\cA$
is nowhere negative. 

Conversely, if $0\notin\Conv\cA$ then by convexity of $\co\cA$ there exists
a hyperplane $H$ passing through the origin 
so that $\co\cA$ lies {\it strictly} on one
side. 
Let $\nu\in S^{n-1}(1)$ be a unit normal to $H$ that points toward the side not containing $\co\cA$.
By the strictness mentioned above and compactness of $\co\cA$,  for some small $\eps>0$, 
$\co\cA+\eps\nu$  still lies on the same side of $H$
as $\co\cA$. Hence, as in the proof of Claim \ref{supportfnClaim}, 
$h_{\co\cA+\eps\nu}(\nu)\le0$. In other words,
$$
h_{\co\cA}(\nu)=h_{\co\cA+\eps\nu}(\nu)-\eps|\nu|^2\le -\eps,
$$
as claimed.
\epf

\bcor
\lb{inclusionCor}
For $\cF,\cU\subseteq\RR^n$ non-empty finite sets with $0\in\Int \co\cU$ (recall (\ref{coA})),
\beq
\lb{bigequalityEq}
\sup\left\{c\in(0,1)\,:\,\left(-\frac{c}{1-c}\co\cF\right)\cap\co\cU\neq\emptyset\right\}
                =
        \begin{cases}
        1,&0\in \co\cF,\\
        \disp\bigg[{1-\min_{S^{n-1}(1)}\frac{h_{\cF}}{h_{\cU}}}
        \bigg]^{-1},&0\notin \co\cF.
        \end{cases}
\eeq
\ecor

\bpf
By assumption $0\in\Int\co\cU$, so by Claim
 \ref{supportfnClaim}, $h_{\cU}|_{S^{n-1}(1)}>0$.
By \eqref{gcEq},
\beq
\lb{gc2Eq}
g_c=h_{\Conv\cU}\cdot\left(1-c\left(1-\frac{h_{\Conv\cF}}{h_{\Conv\cU}}\right)\right).
\eeq

\medskip
\noindent
{\it Case 1: $0\in\Conv\cF$.}
If $0\in\Conv\cF$, i.e., $h_{\Conv\cF}\geq0$, then $g_c\geq h_{\Conv\cU}\cdot\left(1-c\right)>0$ for any $c\in(0,1)$. By \eqref{gc2Eq} and Claim
 \ref{supportfnClaim} then $0\in\Int(c\Conv\cF+(1-c)\Conv\cU)\subseteq c\Conv\cF+(1-c)\Conv\cU$ and so 
\eqref{bigequalityEq} holds in this case.

\medskip
\noindent
{\it Case 2: $0\not\in\Conv\cF$.}
By assumption and Claim  \ref{supportfn2Claim}, $h_{\Conv\cF}<0$ somewhere, by continuity/compactness let $\nu_0$ be a minimizer of the function
$$
\frac{h_{\Conv\cF}}{h_{\Conv\cU}}
$$
on $S^{n-1}(1)$. In particular,
$$
\frac{h_{\Conv\cF}\left(\nu_0\right)}{h_{\Conv\cU}\left(\nu_0\right)}=\min_{S^{n-1}(1)}\frac{h_{\cF}}{h_{\cU}}<0.
$$
Set
$$
c\left(\cF,\cU\right):=\left[1-\frac{h_{\Conv\cF}\left(\nu_0\right)}{h_{\Conv\cU}\left(\nu_0\right)}\right]^{-1}=\left[1-\min_{S^{n-1}(1)}\frac{h_{\cF}}{h_{\cU}}\right]^{-1}.
$$

First, if $c\in(0,c(\cF,\cU))$, for all $\nu\in S^{n-1}(1)$,
\begin{align*}
g_c\left(\nu\right)&=h_{\Conv\cU}\left(\nu\right)\left(1-c\left(1-\frac{h_{\Conv\cF}\left(\nu\right)}{h_{\Conv\cU}\left(\nu\right)}\right)\right)\\
&\geq h_{\Conv\cU}\left(\nu\right)\left(1-c\left(1-\frac{h_{\Conv\cF}\left(\nu_0\right)}{h_{\Conv\cU}\left(\nu_0\right)}\right)\right)\\
&=h_{\Conv\cU}\left(\nu\right)\left(1-\frac{c}{c(\cF,\cU)}\right)>0.
\end{align*}
Thus, by \eqref{gcEq}
and Claim  \ref{supportfnClaim}, $0\in\Int(c\Conv\cF+(1-c)\Conv\cU)\subseteq c\Conv\cF+(1-c)\Conv\cU$; in particular,
$$
\Big(-\frac{c}{1-c}\co\cF\Big)\cap 
        \co\cU\neq\emptyset.
$$
So,
$$\sup\left\{c\in(0,1)\,:\,\Big(-\frac{c}{1-c}\co\cF\Big)\cap 
        \co\cU\neq\emptyset\right\}\ge c(\cF,\cU).
$$

Second, if $c\in(c(\cF,\cU),1)$,
\begin{align*}
g_c\left(\nu_0\right)&=h_{\Conv\cU}\left(\nu_0\right)\left(1-c\left(1-\frac{h_{\Conv\cF}\left(\nu_0\right)}{h_{\Conv\cU}\left(\nu_0\right)}\right)\right)\\
&=h_{\Conv\cU}\left(\nu_0\right)\left(1-\frac{c}{c(\cF,\cU)}\right)<0.
\end{align*}
Thus, by Claim  \ref{supportfn2Claim}, $0\in\RR^n\setminus(c\Conv\cF+(1-c)\Conv\cU)$, equivalently,
$$
\Big(-\frac{c}{1-c}\co\cF\Big)\cap 
        \co\cU=\emptyset,
$$
and
$$
\sup\left\{c\in(0,1)\,:\,\Big(-\frac{c}{1-c}\co\cF\Big)\cap 
        \co\cU\neq\emptyset\right\}\le c(\cF,\cU).
$$
Thus, also in this case \eqref{bigequalityEq} holds.
\epf

The following corollary follows from the proof of Corollary \ref{inclusionCor}.
\bcor
\lb{c0Cor}
For $\cF,\cU\subseteq\RR^n$ non-empty finite sets with $0\in\Int \co\cU$ (recall (\ref{coA})),
let
\begin{align*}
c(\cF,\cU)
:&=\sup\left\{c\in(0,1)\,:\,\left(-\frac{c}{1-c}\co\cF\right)\cap\co\cU\neq\emptyset\right\}
\cr
&=\sup\big\{c\in(0,1)\,:\,0\in (1-c)\co\cU+{c}\co\cF\big\}.
\end{align*}
Recall \eqref{gcEq}. If $c<c(\cF,\cU)$, then $g_c>0$; if $c>c(\cF,\cU)$, then $g_c<0$ somewhere.
\ecor
We can now complete the proof of Proposition \ref{generalizeSong}.
Set
$$
c(\cF,\cU):=
\sup\left\{c\in(0,1)\,:\,\Big(-\frac{c}{1-c}\co\cF\Big)\cap 
        \co\cU\neq\emptyset\right\}.
$$
By assumption $0\in\Int\co\cU$, thus $c(\cF,\cU)>0$.
We consider two cases.

\medskip
\noindent
{\it Case 1: $c\in(0,c(\cF,\cU))$.}
By the proof of Corollary \ref{inclusionCor}, $g_c|_{S^{n-1}(1)}>0$.
By continuity/compactness it attains its minimum $g_c(\hat{\nu})>0$. Therefore,
\beq
\lb{fcnuintEq}
f_c(\nu)\leq\int_0^\infty e^{-krg_c(\nu)}r^{n-1}dr\leq\int_0^\infty e^{-krg_c(\hat{\nu})}r^{n-1}dr<\infty,    
\eeq
    so by \eqref{integrationpolarcoordEq}, $c_k({\cF,\cU})\geq c$, i.e., $c_k({\cF,\cU})\geq c(\cF,\cU)$.
 
\medskip   
\noindent
{\it Case 2: $c\in(c(\cF,\cU),1)$.}
By the proof of Corollary \ref{inclusionCor},
    $g_c(\nu_0)<0$ for some open neighborhood $U\subset S^{n-1}(1)$. Thus, by \eqref{twosidedestimatecUcF}, for some constant
    $C=C(c,k,\cF,\cU)>0$,
    $$
        f_c(\nu)\geq C\int_0^\infty e^{-krg_c(\nu)}r^{n-1}dr=\infty, \q
        \hbox{ for $\nu\in U$},
    $$
    so $c_k({\cF,\cU})\leq c$, i.e.,  $c_k({\cF,\cU})\leq c(\cF,\cU)$.

    In conclusion, $c_k({\cF,\cU})=c(\cF,\cU)$ and Proposition \ref{generalizeSong} is proved.
\end{proof}

\subsection{Complex singularity exponents}

Now, we go back to our setting of estimating complex
singularity exponents, where $P$ was a reflexive polytope coming
from a fan.
For any $k\in\NN$ and a non-empty subset $\cF\subseteq P\cap M/k$
set    
\beq
\lb{ckcFEq}
       c_k(\cF):=\sup\left\{c\in(0,1)\,:\,\int_X\left(\sum_{u\in \cF}\left|s_{k,u}\right|_{h_k}^2\right)^{-\frac{c}{k}}d\mu<\infty\right\}.
\eeq

\bremark
Consider
$$
\widetilde{c_k}(\cF):=\sup\left\{c\in(0,\infty)\,:\,\int_X\left(\sum_{u\in \cF}\left|s_{k,u}\right|_{h_k}^2\right)^{-\frac{c}{k}}d\mu<\infty\right\}.
$$
Then, by definition \eqref{lctSigmaEq},
 $\widetilde{c_k}(\cF)=k\,\lct\big|\Span\{s_{k,u}\}_{u\in\cF}\big|$.
However,  Proposition \ref{generalizeSong} would not be applicable then.
The point is that for the proof of Tian's conjecture we need to compute
equivariant {\it global} log canonical thresholds and there will always be
`admissible' invariant linear series for which $k$ times the log canonical threshold
is at most 1 (i.e., the ``worst" ones will not have $\widetilde{c_k}>1$), i.e., for which ${c_k}=\widetilde{c_k}$. This is essentially
because our $H$ are linear groups so $\{0\}$ is always an $H$-orbit
and $\lct(s_{k,0})=1/k=c_k(\{0\})=\widetilde{c_k}(\{0\})$ (by Lemma \ref{AuxOneSecLCTLemma}).

\eremark

\begin{proposition}\label{Orbit lct}
    For any $k\in\NN$ and a non-empty subset $\cF\subseteq P\cap M/k$,
    $$
        c_k(\cF)=\sup\left\{c\in(0,1)\,:\,P\cap\Big(-\frac{c}{1-c}\co\cF\Big)\neq\emptyset\right\}>0.
    $$
    In particular, it is independent of $k$.
\end{proposition}
\begin{proof}
By \eqref{canonicalomeganEq},
    $$
        \int_X\left(\sum_{u\in \cF}\left|s_{k,u}\right|_{h_k}^2\right)^{-\frac{c}{k}}d\mu=(2\pi)^n\int_{\RR^n}\frac{\left(\sum_{u\in \cF}e^{\langle ku,x\rangle}\right)^{-\frac{c}{k}}}{\left(\sum_{u\in P\cap M/k}e^{\langle ku,x\rangle}\right)^{\frac{1-c}{k}}}dx.
    $$
    By Proposition \ref{generalizeSong} with
    $\cU=P\cap M/k$,
    $$
        c_k(\cF)=\sup\left\{c\in(0,1)\,:\,P\cap
        \Big(-\frac{c}{1-c}\co\cF\Big)\neq\emptyset\right\}
        >0,
    $$
    since $\co\cU=P$ (as the vertices of $P$ belong to $M/k$ for all $k\in\NN$).
\end{proof}

\bremark
Although not needed for our analysis, it might be of independent interest
to understand whether the supremum in \eqref{ckcFEq} (and in
the definition of $\alpha_{k,G(H)}$ \eqref{alpha_{k,G}}) is attained.
In other words, is (recall the notation of Corollary \ref{c0Cor})
\beq
\lb{fc0finiteEq}
\int_{S^{n-1}(1)}f_c(\nu)dS(\nu)
\eeq
finite for $c=c(\cF,\cU)$?
By \eqref{integrationpolarcoordEq}--\eqref{twosidedestimatecUcF}
it is equivalent to consider this question for the integral
\beq
\lb{supportpolarEq}
\int_{S^{n-1}(1)}\int_0^\infty e^{-krg_{c}(\nu)}r^{n-1}drdS(\nu)
=
\int_{\RR^n}e^{-kh_{(1-c)\co\cU+c\co\cF}(x)}dx.
\eeq
A classical formula in convex geometry says that whenever
$K\subset \RR^n$ is a compact convex set with the origin in its
interior, then $n!|K^\circ|=\int_{\RR^n}e^{-h_K(y)}dy$
(for a detailled proof see \cite[(4.2)]{BMR}). In fact,
it is possible to show that when $0\in \partial K$ the formula
still holds in the sense that both sides are equal to $\infty$
(this is related to, but stronger than, the classical fact
that $0\in\Int K$ if and only if $K^\circ$ is bounded
\cite[Corollary 14.5.1]{Rock}).
To summarize, both \eqref{fc0finiteEq} and \eqref{supportpolarEq}
are infinite when $c=c(\cF,\cU)$. We remark that one can also use 
this point of view to give an alternative, but equivalent, 
proof of Corollary \ref{c0Cor}.
\eremark

\subsection{Group orbits and singularities}
\lb{GroupOrbSubSec}
Consider the orbit-averaging map $\pi_H$ \eqref{definition of p},
taking a point to the average of its image under $H$. In particular, points on the same orbit have the same image. 
Let $M_\RR^H$ be the subspace of fixed points of $H$ in $M_\RR$. Then $\pi_H|_{M_\RR^H}=id_{M_\RR^H}$, and $M_\RR^H$ is the image of $\pi_H$. 
Note that unless $H$ is trivial, this is a proper subspace of $\MM_\RR$,
so $\pi_H$ is a projection from $M_\RR$ to the invariant subspace $M_\RR^H$,
justifying the notation.

Recall \eqref{PHEq},
\begin{equation*}
    P^H:=P\cap M_\RR^H,
\end{equation*}
Since $P$ is convex, 
$\pi_H(P)\subseteq P$. Hence $\pi_H(P)\subseteq P^H$. On the other hand, $P^H=\pi_H(P^H)\subseteq \pi_H(P)$. This implies
\begin{equation}\label{PHpiHPEq}
P^H=\pi_H(P).
\eeq
Then (recall \eqref{VerAEq})
$$
P^H=\pi_H(P)=\pi_H(\Conv(\Ver P))=\Conv(\pi_H(\Ver P)).
$$
In particular,
\begin{equation}\label{Vert P^H}
    \Ver P^H\subseteq \pi_H(\Ver P).
\end{equation}
Note that equality does not hold in general (for instance, consider
$P=[-1,1]^2$ and $H$ the reflection group about a diagonal).

\begin{lemma}\label{Orbit lct simplified}
    Fix $u_0\in M^H_\RR$, and let $\cF$ be a non-empty $H$-invariant convex subset of some fiber $\pi_H^{-1}(u_0)$ of $\pi_H$. Then $\cF\cap P\neq\emptyset$ if and only if $u_0\in P$.
\end{lemma}
\begin{proof}
    If $\cF\cap P\neq\emptyset$, 
    then $\pi_H(\cF\cap P)\neq\emptyset$.
    By \eqref{PHpiHPEq}, $\pi_H(\cF\cap P)\subset\pi_H(P)=P^H\subset P$, and since
    $\emptyset\not=\cF\subset \pi_H^{-1}(u_0)$, then 
    $\pi_H(\cF\cap P)\subset\pi_H(\cF)=\{u_0\}$.
    Thus, $\pi_H(\cF\cap P)\subset\{u_0\}\cap P$
    and for this to be nonempty it follows that $u_0\in P$.
    
    For the converse,
    since $\cF$ is $H$-invariant and convex, $\pi_H(\cF)\subset \cF$.
    Additionally, as before
    $\emptyset\not=\cF\subset \pi_H^{-1}(u_0)$
    implies
    $\pi_H(\cF)=\{u_0\}$. Thus $u_0\in\cF$.
    Therefore, if $u_0\in P$ then actually $u_0\in \cF\cap P$.
    \end{proof}

\subsection{Proof of Tian's conjecture}

We can now prove our main result.

\begin{proof}[Proof of Theorem \ref{alpha_k,Gformula}]
Fix $k\in\NN$ and let $O^{(k)}_1,\ldots,O^{(k)}_N$ be the orbits of the action of $H$ on $P\cap M/k$. Recall that $\pi_H$ maps an orbit to a singleton,
    $$
    \big\{o^{(k)}_i\big\}:=\pi_H\big(O^{(k)}_i\big).
    $$
    Note that 
    $O^{(k)}_i\subset\pi_H^{-1}\big(o^{(k)}_i\big)$
    and hence (as the action of $H$ and hence also $\pi_H$ is linear) also
    $\co O^{(k)}_i\subset\pi_H^{-1}\big(o^{(k)}_i\big)$.
    Moreover, $\co O^{(k)}_i$ is convex and $H$-invariant.
    By Corollary \ref{alpha orbit exp cor}, Proposition \ref{Orbit lct}, and Lemma \ref{Orbit lct simplified},
    \begin{align*}
        \alpha_{k,G(H)}
        &=\min_{1\leq i\leq N}c\big(O^{(k)}_i\big)\\
        &=\min_{1\leq i\leq N}\sup\left\{c\in(0,1)\,:\,\Big(-\frac{c}{1-c}\co O^{(k)}_i\Big)\cap P\not=\emptyset\right\}
        \\
        &=\min_{1\leq i\leq N}\sup\left\{c\in(0,1)\,:\,-\frac{c}{1-c}o^{(k)}_i\in P\right\}\\
        &=\sup\left\{c\in(0,1)\,:\,-\frac{c}{1-c}\big\{o^{(k)}_1,\ldots,o^{(k)}_N\big\}\subset P\right\}.
    \end{align*}
    Since
    $$
        \big\{o^{(k)}_1,\ldots,o^{(k)}_N\big\}=\pi_H\left(\bigcup_{i=1}^NO^{(k)}_i\right)=\pi_H
        \left({k}^{-1}M\cap P\right),
    $$
    then
    \begin{align}
        \alpha_{k,G(H)}&=\sup\left\{c\in(0,1)\,:\,-\frac{c}{1-c}\pi_H({k}^{-1}M\cap P)\subset P\right\}\nonumber\\
        &=\sup\left\{c\in(0,1)\,:\,-\frac{c}{1-c}\Conv(\pi_H({k}^{-1}M\cap P))\subset P\right\}\nonumber\\
        &=\sup\left\{c\in(0,1)\,:\,-\frac{c}{1-c}\pi_H(\Conv({k}^{-1}M\cap P))\subset P\right\}\nonumber\\
        &=\sup\left\{c\in(0,1)\,:\,-\frac{c}{1-c}\pi_H(P)\subset P\right\}\nonumber\\
        &=\sup\left\{c\in(0,1)\,:\,-\frac{c}{1-c}P^H\subset P\right\}\nonumber\\
        &=\sup\left\{c\in(0,1)\,:\,-\frac{c}{1-c}\Ver P^H\subset P\right\}\nonumber\\
        &=\min_{u\in\Ver P^H}\sup\left\{c\in(0,1)\,:\,-\frac{c}{1-c}u\in P\right\}.\label{alpha_{k,G} using Ver(P^H)}
    \end{align}
    Also, by \eqref{Vert P^H},
    \begin{align*}
        \alpha_{k,G(H)}&=\sup\left\{c\in(0,1)\,:\,-\frac{c}{1-c}\Ver P^H\subset P\right\}\\
        &\geq\sup\left\{c\in(0,1)\,:\,-\frac{c}{1-c}\pi_H(\Ver P)\subset P\right\}\\
        &\geq\sup\left\{c\in(0,1)\,:\,-\frac{c}{1-c}\pi_H(P)\subset P\right\}\\
        &=\alpha_{k,G(H)}.
    \end{align*}
    Therefore,
    \begin{align}
        \alpha_{k,G(H)}&=\sup\left\{c\in(0,1)\,:\,-\frac{c}{1-c}\pi_H(\Ver P)\subset P\right\}\nonumber\\
        &=\min_{u\in \pi_H(\Ver P)}\sup\left\{c\in(0,1)\,:\,-\frac{c}{1-c}u\in P\right\}.\label{alpha_{k,G} using p(Ver(P))}
    \end{align}
    Recall \eqref{PDef2Eq},
    $$
        P=\bigcap_{j=1}^d\{x\in M_\RR~|~\langle x,-v_j\rangle\leq1\}.
    $$
    For any point $u$ and $c\in(0,1)$,
    $$
        -\frac{c}{1-c}u\in P\q
    \hbox{
    if and only if }\q
            c\leq\min_j\frac{1}{\langle u,v_j\rangle+1}\q
            \hbox{for all $j\in\{1,\ldots,d\}$}.
    $$
    In sum, combining \eqref{alpha_{k,G} using Ver(P^H)} and \eqref{alpha_{k,G} using p(Ver(P))} implies \eqref{MainTheoremEq}.
    In particular, $\alpha_{k,G(H)}$ is independent of $k\in\NN$ and, by Corollary \ref{alpha_G corollary}, is equal to $\alpha_{G(H)}$.
This completes the proof of Theorem \ref{alpha_k,Gformula}.
\end{proof}

\section{Grassmannian Tian invariants}
\lb{alphakmSec}

The following definition is due to Tian \cite[(6.1)]{Tian91}. 

\begin{definition}
\lb{alphakmGDef}
    Let $X$ be a Fano manifold and $G\subseteq\Aut X$ a compact subgroup. For $k,m\in\NN$, define (recall \eqref{lctSigmaEq})
    $$
        \alpha_{k,m,G}:=k\inf_{\substack{|V|\subset|-kK_X|\\\dim V=m\\V^G=V}}\lct|V|,
    $$
    with the convention $\inf\emptyset=\infty$.
\end{definition}

Equivalently, these invariants are the minimum of the function 
$V\mapsto k\lct|V|$ over the Grassmannian restricted to the 
subset of $G$-invariant $m$-dimensional 
subspaces of $H^0(X,-kK_X)$. This subset can be, in some situations, even
a finite set (see Remark \ref{NoMonotonicityInmRemark}).
By Remark \ref{GenAlphakmRem} below, these invariants generalize the invariants
$\alpha_{k,G}$. Note also that Definition \ref{alphakmGDef} is an {\it algebraic}
one. Also, $\alpha_{k,E_P(k),G}=\infty$ (recall \eqref{EPkEq}) so it is {\it not} equal to
$\alpha_{k,G}$ from Definition \ref{alphakGhmuDef}.

In the toric setting it is natural to consider Conjecture \ref{TianGenConj}
for the invariants $\alpha_{k,m,(S^1)^n}$. For these invariants
we show that Conjecture \ref{TianGenConj} does not quite hold,
and we give a precise breakdown of when it does in terms of a
natural convex geometric criterion.
In contrast with previous sections, we do not work with the additional
symmetry corresponding to the groups $G(H)$  \eqref{GHEq}. The reason for
that is that even in simple situations $\alpha_{k,m,G(H)}$ will not be
well-behaved: see Example \ref{alphakmGHExample}.

\begin{remark}
\lb{GenAlphakmRem}
    Putting $H=\{\id\}$ in Proposition \ref{analytic alpha_k,G}, 
    notice that $\alpha_{k,(S^1)^n}=\alpha_{k,1,(S^1)^n}$. In this sense, $\alpha_{k,m,(S^1)^n}$ is a generalization of the $\alpha_{k,(S^1)^n}$-invariants.
\end{remark}

First, we show, as a rather immediate consequence of our work, an explicit formula for
the invariants of Definition \ref{alphakmGDef}.

\bcor
\label{alpha_k,m,G lemma}
    Let $X$ be toric Fano with associated polytope $P$ (\ref{PDef2Eq}). For $k,m\in\NN$
    (recall (\ref{ckcFEq})),
    $$
        \alpha_{k,m,(S^1)^n}=\min_{\substack{\cF\subseteq P\cap M/k\\\left|\cF\right|=m}}c_k\left(\cF\right).
    $$
\ecor
\begin{proof}
    Fix $k,m\in\NN$. Let $V\subseteq H^0(X,-kK_X)=\Span\{s_u\}_{u\in P\cap M/k}$ be a 
    subspace invariant under $(S^1)^n$. By Lemma \ref{onedimpiecesLema}, 
    there exists a unique  $\cF\subseteq P\cap M/k$ such that
    $$
        V=\Span\left\{s_u\right\}_{u\in\cF}.
    $$
    Note that $|\cF|=\dim V$. By \eqref{lctSigmaEq} and \eqref{ckcFEq},
    $$
        \alpha_{k,m,(S^1)^n}=k\min_{\substack{\cF\subseteq P\cap M/k\\\left|\cF\right|=m}}\lct\left|\Span\left\{s_u\right\}_{u\in\cF}\right|=\min_{\substack{\cF\subseteq P\cap M/k\\\left|\cF\right|=m}}c_k\left(\cF\right).
    $$
\end{proof}

\begin{lemma}\label{norm of set}
    Let $K\subset\RR^n$ be a compact convex set with $0\in\Int K$. For any set $S\subset\RR^n$,
    $$
        \inf_{S}\left\|\,\cdot\,\right\|_K=\inf\left\{\lambda\geq0\,:\,S\cap\lambda K\neq\emptyset\right\}.
    $$
\end{lemma}
\begin{proof}
    If $S\cap\lambda K\neq\emptyset$ for some $\lambda\geq0$, pick $x\in S\cap\lambda K$. Since $x\in\lambda K$, by \eqref{normdef}, $\|x\|_K\leq\lambda$. Therefore $\inf_{x\in S}\|x\|_K\leq\inf\{\lambda\geq0\,:\,S\cap\lambda K\neq\emptyset\}$.

    On the other hand, for any $x\in S$, $S\cap\|x\|_KK\supseteq\{x\}\neq\emptyset$
    by \eqref{infachievedEq}. So $\inf_{x\in S}\|x\|_K\geq\inf\{\lambda\geq0\,:\,S\cap\lambda K\neq\emptyset\}$.
\end{proof}

Proposition \ref{Orbit lct} can be reformulated in terms of near-norms:
\bcor\label{c_k norm}
    For $k\in\NN$ and a non-empty set $\cF\subseteq P\cap M/k$ (recall (\ref{ckcFEq})),
    $$
        c_k\left(\cF\right)=\frac{1}{1+\min\limits_{\Conv\cF}\left\|\,\cdot\,\right\|_{-P}}.
    $$
\ecor
\begin{proof}
    By Proposition \ref{Orbit lct} and Lemma \ref{norm of set}, 
    \begin{align*}
        c_k\left(\cF\right)&=\sup\left\{c\in(0,1)\,:\,P\cap\left(-\frac{c}{1-c}\co\cF\right)\neq\emptyset\right\}
        \\
        &=\sup\left\{c\in(0,1)\,:\,\Conv\cF\cap\frac{1-c}{c}\left(-P\right)\neq\emptyset\right\}
        \\
        &=\sup\left\{\frac1{1+a}\,:\,a\in(0,\infty),\;\Conv\cF\cap \left(-aP\right)\neq\emptyset\right\}
        \\
        &=\Big[1+\inf\Big\{a\in(0,\infty)\,:\,\Conv\cF\cap \left(-aP\right)\neq\emptyset\Big\}
        \Big]^{-1}
        \\
        &=\big[1+\inf_{\Conv\cF}\left\|\,\cdot\,\right\|_{-P}
        \big]^{-1}
        =\big[1+\min_{\Conv\cF}\left\|\,\cdot\,\right\|_{-P}
        \big]^{-1},        
    \end{align*}
    using compactness of $\Conv\cF$.
\end{proof}

Combining Corollaries \ref{alpha_k,m,G lemma} and  \ref{c_k norm} yields an explicit formula for $\alpha_{k,m,(S^1)^n}$.
\begin{corollary}
\lb{alphakmtorusCor}
    Let $X$ be toric Fano with associated polytope $P$ (\ref{PDef2Eq}). For $k,m\in\NN$,
    $$
        \alpha_{k,m,(S^1)^n}=
        \bigg[{1+\max\limits_{\substack{\cF\subseteq P\cap M/k\\\left|\cF\right|=m}}\;\min\limits_{\Conv\cF}\left\|\,\cdot\,\right\|_{-P}}
        \bigg]^{-1}
        .
    $$
\end{corollary}

\begin{remark}
    For $m=1$, Corollary \ref{alphakmtorusCor} reads
    $$
        \alpha_{k,(S^1)^n}=\frac{1}{1+\max\limits_{u\in P\cap M/k}\left\|u\right\|_{-P}}.
    $$
    Being a near-norm, the function $\|\cdot\|_{-P}$ satisfies the triangle inequality and is positively 1-homogeneous, hence it is convex.
    By Lemma \ref{PIntegralLem},
    the vertex set of $P$ is in the integral
    lattice $M$, hence also contained in $M/k$ for all $k\in\NN$. 
    This combined with Lemma \ref{ConvMaxAttaindVertexLem} gives,
    \beq
    \lb{anotheralphakconstEq}
        \alpha_{k,(S^1)^n}=\frac{1}{1+\max\limits_{\Ver{P}}\left\|\,\cdot\,\right\|_{-P}}=\frac{1}{1+\max\limits_{ P}\left\|\,\cdot\,\right\|_{-P}}.
    \eeq
    This is, of course, a special case of Theorem \ref{alpha_k,Gformula} (proved in \S\ref{OrbitsEstimateSec}), derived here in slightly different notation (using
    the near-norm instead of the support function).
    Thus, 
    \beq
    \lb{alphakBJEq}
    \alpha_{k,(S^1)^n}=\alpha_{(S^1)^n}=\alpha,
    \eeq 
    where the first
    equality follows from \eqref{anotheralphakconstEq} and Demailly's result \cite[(A.1)]{CS08},
    and the last equality
    follows by comparing Theorem \ref{alpha_k,Gformula} with
    Blum--Jonsson's formula \cite[(7.2)]{BJ20}.
\end{remark}

\bprop\label{alpha_k,m,G monotone in m}
    Let $X$ be toric Fano with associated polytope $P$ (\ref{PDef2Eq}). 
    For $k\in\NN$ and $m'> m$,
    \beq
    \lb{alphakmmprimeEq}
        \alpha_{k,m',\left(S^1\right)^n}\geq\alpha_{k,m,\left(S^1\right)^n}.
    \eeq
    In particular,
    $$
        \alpha_{k,m,\left(S^1\right)^n}\geq\alpha_{k,1,\left(S^1\right)^n}=\alpha_{k,\left(S^1\right)^n}=\alpha.
    $$
    Equality holds if and only if there is $\cF\subseteq P\cap M/k$ with $|\cF|=m$ 
    satisfying 
    \beq
    \lb{constantoncoFEq}
        \left\|\,\cdot\,\right\|_{-P}\Big|_{\Conv\cF}
        =
        \max_{ P}\left\|\,\cdot\,\right\|_{-P}.
    \eeq
\eprop

\bremark
\lb{NoMonotonicityInmRemark}
For a  general Fano $X$, it does not follow from Definition \ref{alphakmGDef} that 
$\alpha_{k,m',G}\geq\alpha_{k,m,G}$, i.e., there may be no monotonicity in $m$. For instance, there might be
no $m$-dimensinal $G$-invariant subspaces. Thanks to Lemma \ref{onedimpiecesLema}
the situation for $G=(S^1)^n$ is particularly simple: there are finitely-many
$m$-dimensional $(S^1)^n$-invariant subspaces (in fact, the Grassmannian
of $m$-dimensional subspaces has exactly ${E_P(k)\choose m}$
fixed points by the $(S^1)^n$-action), and moreover every 
$(S^1)^n$-invariant subspace consists of 1-dimensional blocks. 
\eremark

\begin{proof}
    Pick $\cF'\subseteq P\cap M/k$ that computes $\alpha_{k,m',(S^1)^n}$
    (such a subset exists since $P\cap M/k$ is a finite set). That is, 
    $|\cF'|=m'$ and 
    $
        \alpha_{k,m',\left(S^1\right)^n}=\big[{1+\min\limits_{\Conv\cF'}\left\|\,\cdot\,\right\|_{-P}}\big]^{-1}
    $ (recall Corollary \ref{alphakmtorusCor}).
    For $\cF$ be a subset of $\cF'$ with $|
    \cF|=m$, Corollary \ref{alphakmtorusCor} implies
    $$
        \alpha_{k,m,\left(S^1\right)^n}=\frac{1}{1+\max\limits_{\substack{\cA\subseteq P\cap M/k\\\left|\cA\right|=m}}\;\min\limits_{\Conv\cA}\left\|\,\cdot\,\right\|_{-P}}\leq\frac{1}{1+\min\limits_{\Conv\cF}\left\|\,\cdot\,\right\|_{-P}}\leq\frac{1}{1+\min\limits_{\Conv\cF'}\left\|\,\cdot\,\right\|_{-P}}=\alpha_{k,m',\left(S^1\right)^n},
    $$
proving \eqref{alphakmmprimeEq}.

Next, suppose $\alpha_{k,m,(S^1)^n}=\alpha$, i.e.
(recall \eqref{anotheralphakconstEq}--\eqref{alphakBJEq}),
    \begin{equation}\label{equality condition}
        \max_{\substack{\cF\subseteq P\cap M/k\\\left|\cF\right|=m}}
        \;\min_{\Conv\cF}\left\|\,\cdot\,\right\|_{-P}=\max_{ P}\left\|\,\cdot\,\right\|_{-P}.
    \end{equation}
    For any $\cF\subseteq P$, convexity of $P$ (recall (\ref{PDef2Eq})) implies $\Conv\cF\subseteq P$. Hence,
    $$
        \min_{\Conv\cF}\left\|\,\cdot\,\right\|_{-P}\leq\max_{ P}\left\|\,\cdot\,\right\|_{-P}.
    $$
    The equality \eqref{equality condition} holds if and only if there is $\cF\subseteq P\cap M/k$ such that $|\cF|=m$ and
    $$
        \min_{\Conv\cF}\left\|\,\cdot\,\right\|_{-P}=\max_{ P}\left\|\,\cdot\,\right\|_{-P},
    $$
    which forces $\|\cdot\|_{-P}$ to be the constant $\max_{ P}\|\,\cdot\,\|_{-P}$ 
    on $\Conv\cF$.
\end{proof}

\subsection{Proof of Theorem \ref{SecondMainThm} and the intuition behind it}
\lb{ProofThm1.6SubSec}

Theorem \ref{alpha_k,Gformula} resolved Tian's classical stabilization Problem \ref{TianProb}
in the affirmative. This corresponds to the case $m=1$
of Conjecture \ref{TianGenConj}.
Next we show that Conjecture \ref{TianGenConj} is only partially true 
for $m\ge2$.

\begin{proof}[Proof of Theorem \ref{SecondMainThm}]
    Let $\cF\subset P\cap M/k$ with $|\cF|=m$.
    By convexity of $P$ and as $m\ge2$,  $\Conv\cF$ contains an interval in $P$.
    In particular, $\cF$    contains a point of $P\setminus\Ver P$.
Thus, if \eqref{starPEq} holds then \eqref{constantoncoFEq} does not
hold (for any such $\cF$), and
Proposition \ref{alpha_k,m,G monotone in m} implies \eqref{TianGenConjFailsEq}.

Next, suppose \eqref{starPEq} fails (and now we relax to $m\ge1$).
    Let $\hat x\in P\setminus\Ver P$ achieve the maximum of $\|\cdot\|_{-P}$ on $P$.
    Express $\hat{x}$ as a convex combination of a {\it subset} of vertices of $P$,
    $\{p_1,\ldots,p_\ell\}\subset\Ver P$, namely, 
    \beq
    \lb{hatxEq}
        \hat{x}=\sum_{i=1}^{\ell}\hat{\lambda}_ip_i, \q \sum_{i=1}^{\ell}\hat{\lambda}_i=1,
        \q \hat{\lambda}\in(0,1)^\ell.
    \eeq
Note that necessarily 
\beq
\lb{ellbigEq}
\ell>1.
\eeq
We claim that
    $$
        \left\|x\right\|_{-P}=\max_{x\in P}\left\|\,\cdot\,\right\|_{-P}, \q
        \hbox{for any $x\in\Conv\{p_i\}_{i=1}^\ell$}.
    $$
    Indeed, letting
    $$
P':=\Conv\{p_i\}_{i=1}^\ell\subset P,
    $$
    the claim follows from $\hat x\in\Int P'\subset P$ and
    Lemma \ref{ConvMaxAttaindVertexLem} (iii).
    
Now, the lattice polytope ${P}':=\Conv\{p_i\}_{i=1}^\ell$ has positive dimension
by \eqref{ellbigEq}. Equivalently, the Ehrhart polynomial of ${P}'$ has degree at least 1. Hence, by Proposition \ref{Ehrhart polynomial},
$|{P}'\cap M/k|\geq m$ for all sufficiently large $k$. 
Pick $\cF\subseteq {P}'\cap M/k\subseteq P\cap M/k$ such that $|\cF|=m$. Then $\Conv\cF\subseteq{P}'$. By Proposition \ref{alpha_k,m,G monotone in m}, \eqref{TianGenConjHoldsEq} follows. 
The case $m=1$ was
already treated in Theorem \ref{alpha_k,Gformula} but also from this proof
we see that any $k\in\NN$ works for $m=1$.
\end{proof}

The proof of Theorem \ref{SecondMainThm} and condition \eqref{starPEq} 
have a pleasing geometric interpretation.
Geometrically, by Proposition \ref{alpha_k,m,G monotone in m}, 
$\alpha_{k,m,(S^1)^n}=\alpha$ if and only if there is a subset $\cF\subseteq P\cap M/k$ such that $|\cF|=m$ and $\Conv\cF$ lies entirely in the level set 
\beq
\lb{maxlevelsetEq}
\Big\{\|\,\cdot\,\|_{-P}=\max_P\|\cdot\|_{-P}\Big\}\subseteq\partial P
\eeq 
Recall that any convex subset of the boundary of a polytope must lie in a single face. Hence $\cF$ must be a subset of a face $F$ of $P$. We may assume $F$ is the minimal face that contains $\cF$, i.e., $\cF$ intersects the relative interior of $F$. In particular,  $\|\cdot\|_{-P}|_F$ 
attains its maximum in the interior of $F$ and so $F$ has to lie entirely in the level set \eqref{maxlevelsetEq} by Lemma \ref{ConvMaxAttaindVertexLem}. Therefore $\alpha_{k,m,(S^1)^n}=\alpha$ if and only if there is a subset $\cF\subseteq P\cap M/k$ such that $|\cF|=m$ and $\cF$ is a subset of a face that lies in the level set \eqref{maxlevelsetEq}, or equivalently, there is a face $F$ that lies in the level set \eqref{maxlevelsetEq}, and $|F\cap M/k|\geq m$.

Notice that the level set 
$\{\|\,\cdot\,\|_{-P}=\lambda\}$
is the dilation 
$\lambda\partial(-P)$, and \eqref{maxlevelsetEq} 
is the largest dilation that intersects $P$ (by Lemma \ref{norm of set});
see Figure \ref{*P example}. Combining all the above then, \eqref{starPEq} states that 
$$
\hbox{$P$ intersects $\max_P\|\cdot\|_{-P}\partial(-P)$ only at vertices.}
$$
If \eqref{starPEq} holds, then any such face $F$ consists of a single lattice point. In particular, $|F\cap M/k|=1$ for any $k$, and $\alpha_{k,m,(S^1)^n}$ stabilizes if and only if $m=1$. On the other hand, if \eqref{starPEq} fails, then there is a face $F$ of positive dimension that lies in the level set 
\eqref{maxlevelsetEq}.
Since $|F\cap M/k|$ increases to infinity, for any $m$, we can find $k_0$ such that $|F\cap M/k|\geq m$ for $k\geq k_0$, i.e., $\alpha_{k,m,(S^1)^n}$ stabilizes
in $k$.

\subsection{A Demailly result for Grassmannian invariants}

The next result was obtained independently by Li--Zhu \cite[Proposition 4.1]{LiZhu}.

\bprop
\lb{convergencalphakmProp}
    Let $X$ be toric Fano with associated polytope $P$ (\ref{PDef2Eq}). For $m\in\NN$,
    $$
        \lim_{k\to\infty}\alpha_{k,m,\left(S^1\right)^n}=\alpha.
    $$
\eprop
\begin{proof}
    By continuity, given $\eps>0$ there exists an open set $U_\eps\subset P$ such that
    \begin{equation}\label{epsilon argument}
        \left\|\,\cdot\,\right\|_{-P}\Big|_{U_\eps}>\max_{P}\left\|\,\cdot\,\right\|_{-P}-\eps.
    \end{equation}
    Fix $\delta>0$ sufficiently small so that there is a closed cube $C_\delta\subset U_\eps$ with edge length $\delta$. Now, for $k\geq K$, $|C_\delta\cap M/k|\geq\lfloor K\delta\rfloor^n$. Given $m\in\NN$,  choose  $K$ so $\lfloor K\delta\rfloor^n\geq m$.
    Then for each $k\ge K$ there exists $\cF\subset C_\delta\cap M/k\subset P\cap M/k$ such that $|\cF|=m$. Since $\cF\subset C_\delta$ and $C_\delta$ is convex, $\Conv\cF\subset C_\delta\subset U_\eps$. By \eqref{epsilon argument},
    $$
        \min_{\co\cF}\left\|\,\cdot\,\right\|_{-P}\geq\min_{C_\delta}\left\|\,\cdot\,\right\|_{-P}>\max_{P}\|\,\cdot\,\|_{-P}-\eps.
    $$
    So by Corollary \ref{alphakmtorusCor} and 
    \eqref{anotheralphakconstEq}--\eqref{alphakBJEq},
    $$
        \alpha_{k,m,\left(S^1\right)^n}\leq\frac{1}{1+\min\limits_{\Conv\cF}\left\|\,\cdot\,\right\|_{-P}}<\frac{1}{1+\max\limits_{ P}\|\,\cdot\,\|_{-P}-\eps}=\frac{\alpha}{1-\eps\alpha},
    $$
    for any $k$ such that $\lfloor k\delta\rfloor\geq {\root n \of m}$. Hence
    $
        \limsup_{k\to\infty}\alpha_{k,m,\left(S^1\right)^n}\leq\frac{\alpha}{1-\eps\alpha}.
    $
    On the other hand, by Proposition \ref{alpha_k,m,G monotone in m},
    $
        \liminf_{k\to\infty}\alpha_{k,m,\left(S^1\right)^n}\geq\alpha.
    $
    Since $\eps>0$ is arbitrary the proof is complete.
\end{proof}

\section{Examples}
\lb{ExamplesSec}

This section illustrates Theorems \ref{alpha_k,Gformula} and \ref{SecondMainThm}
in the case of toric del Pezzo surfaces. 

\begin{example}
\lb{P2ExampleFirstExam}
    Let $v_1=(1,0),v_2=(0,1),v_3=(-1,-1)$. Then $\PP^2=X(\Delta)$ for
    $$
\Delta=\{\{0\}, \RR_{+}v_1,\RR_{+}v_2,\RR_{+}v_3,\RR_{+}v_1+\RR_{+}v_2,
\RR_{+}v_1+\RR_{+}v_3,\RR_{+}v_2+\RR_{+}v_3\}
    $$
    and
$P=\{-v_1,-v_2,-v_3\}^\circ=\{y\in M_\RR\,:\, \langle y,-v_i\rangle\le1, i\in\{1,2,3\}\}\subset M_\RR$ is depicted in Figure \ref{alpha_G example}.
    Then $\Aut P\cong D_{6}=S_3$, the dihedral/cyclic group of order 6 consisting of the permutations of the 3 vertices, or the 3 homogeneous coordinates $\PP^2$, as in Figure \ref{P2figure}. 
    For $H=\Aut P$, $P^H$ is the origin, so by Theorem \ref{alpha_k,Gformula}, $\alpha_{k,G(H)}=\alpha_{G(H)}={1}$. 
    When $H$ is the trivial group then $P^H=P$ and 
    Theorem \ref{alpha_k,Gformula} gives 
    $\alpha_{k,G(H)}=\alpha_{G(H)}=\frac{1}{3}$, 
    with minimum achieved at any vertex of $P$.
    When $H$ is the subgroup $\ZZ_2$ given by reflection about, say, $y=x$
    (i.e., generated by the matrix $\begin{pmatrix}0&1\cr1&0\end{pmatrix}$ associated to switching two of the homogeneous coordinates of $\PP^2$),
    $P^H$ is the line segment 
    $\co\{(-1,-1),(\frac{1}{2},\frac{1}{2})\}$
    and
    $\alpha_{k,G(H)}=\alpha_{G(H)}=\frac{1}{3}$, 
    with minimum achieved at $(-1,-1)$.
    When $H$ is the subgroup of order $3$,
    $P^H=\{0\}$, and
    $\alpha_{k,G(H)}=\alpha_{G(H)}=1$, 
    \end{example}

\begin{figure}
        \centering
        \begin{subfigure}{.3\textwidth}
            \begin{tikzpicture}
                \draw[->,gray](-1.5,0)--(2,0)node[below]{$x$};
                \draw[->,gray](0,-1.5)--(0,2)node[left]{$y$};
                \draw(0,0)node[below right]{$0$};
                \draw(2,-1)--(-1,2)--(-1,-1)--cycle;
            \end{tikzpicture}
        \end{subfigure}
        \begin{subfigure}{.3\textwidth}
            \centering
            \begin{tikzpicture}
                \draw[->,gray](-1.5,0)--(2,0)node[below]{$x$};
                \draw[->,gray](0,-1.5)--(0,2)node[left]{$y$};
                \draw(0,0)node[below right]{$0$};
                \draw(2,-1)--(-1,2)--(-1,0)--(0,-1)--cycle;
                \draw[dashed](-.5,-.5)node[below left]{$(-\frac{1}{2},-\frac{1}{2})$}--(.5,.5)node[above right]{$(\frac{1}{2},\frac{1}{2})$};
            \end{tikzpicture}
        \end{subfigure}
        \begin{subfigure}{.3\textwidth}
            \centering
                \begin{tikzpicture}
                \draw[->,gray](-1.5,0)--(1.5,0)node[below]{$x$};
                \draw[->,gray](0,-1.5)--(0,1.5)node[left]{$y$};
                \draw(0,0)node[below right]{$0$};
                \draw(1,-1)--(1,0)--(0,1)--(-1,1)--(-1,-1)--cycle;
                \draw[dashed](-1,-1)node[below left]{$(-1,-1)$}--(.5,.5)node[above right]{$(\frac{1}{2},\frac{1}{2})$};
            \end{tikzpicture}
        \end{subfigure}
        \begin{subfigure}{.3\textwidth}
            \centering
                \begin{tikzpicture}
                \draw[->,gray](-1.5,0)--(1.5,0)node[below]{$x$};
                \draw[->,gray](0,-1.5)--(0,1.5)node[left]{$y$};
                \draw(0,0)node[below right]{$0$};
                \draw(1,-1)--(1,0)--(0,1)--(-1,1)--(-1,0)--(0,-1)--cycle;
            \end{tikzpicture}
        \end{subfigure}
        \begin{subfigure}{.3\textwidth}
            \centering
                \begin{tikzpicture}
                \draw[->,gray](-1.5,0)--(1.5,0)node[below]{$x$};
                \draw[->,gray](0,-1.5)--(0,1.5)node[left]{$y$};
                \draw(0,0)node[below right]{$0$};
                \draw(1,-1)--(1,1)--(-1,1)--(-1,-1)--cycle;
            \end{tikzpicture}
        \end{subfigure}
    \caption{The polytope $P$ for del Pezzo surfaces, namely $\PP^2$ blown up at no more than 3 generically positioned points, and $\PP^1\times\PP^1$. For the two K-unstable examples, the automorphism group $H=\Aut(P)$ is generated by the reflection about $y=x$, and $P^H$ is the intersection of $P$ with this reflection axis. For the other three examples, the automorphism group $H=\Aut(P)$ is the dihedral group associated to the polygon $P$, and $P^H=\{0\}$.}\label{alpha_G example}
\end{figure}
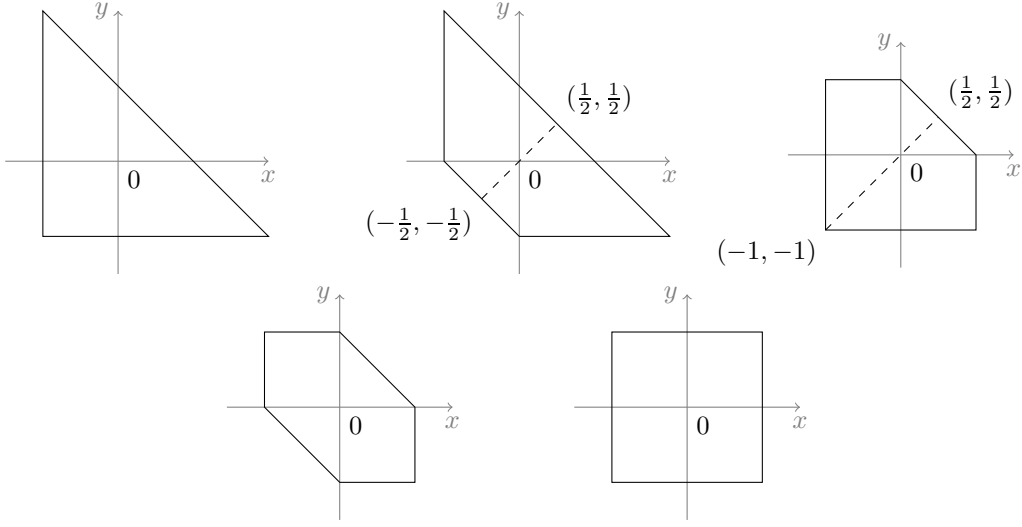

\begin{example}
    Let $P$ be the polytope for $\PP^2$ blown up at one point, i.e.,
    $v_1,v_2,v_3$ as above and $v_4=(1,1)$.
    Then $\Aut P\cong\ZZ_2$, as shown in Figure \ref{alpha_G example}. 
    There are therefore two choices for $H$. When $H=\Aut P$, $P^H$ is the line segment $\co\{(-\frac{1}{2},-\frac{1}{2}),(\frac{1}{2},\frac{1}{2})\}$. By Theorem \ref{alpha_k,Gformula}, $\alpha_{k,G(H)}=\alpha_{G(H)}=\frac{1}{2}$, with minimum achieved at $\pm(\frac{1}{2},\frac{1}{2})$. When $H$ is the trivial group then $P^H=P$ and 
    Theorem \ref{alpha_k,Gformula} gives 
    $\alpha_{k,G(H)}=\alpha_{G(H)}=\frac{1}{3}$, 
    with minimum achieved at either of the vertices $(-1,2)$ or $(2,-1)$. 
    \end{example}

\begin{example}
    Let $P$ be the polytope for $\PP^2$ blown up at two points, i.e.,
    $v_1,v_2,v_3$ as above and $v_4=(-1,0), v_5=(0,-1)$.
    Again $\Aut P\cong\ZZ_2$, as shown in Figure \ref{alpha_G example}. 
    For $H=\Aut P$, $P^H$ is the line segment $\co\{(-1,-1),(\frac{1}{2},\frac{1}{2})\}$.
    By Theorem \ref{alpha_k,Gformula}, $\alpha_{k,G}=\alpha_G=\frac{1}{3}$, with minimum achieved at $(-1,-1)$. When $H$ is the trivial group then $P^H=P$ and 
    Theorem \ref{alpha_k,Gformula} gives 
    $\alpha_{k,G(H)}=\alpha_{G(H)}=\frac{1}{3}$, 
    with minimum achieved still at $(-1,-1)$.
\end{example}

\begin{example}
\lb{P2blowupthreeExam}
    Let $P$ be the polytope for $\PP^2$ blown up at three non-colinear points, i.e.,
    $v_1,\ldots,v_5$ as in the previous example and $v_6=(1,1)$.
    Now $\Aut P\cong D_{12}$, the dihedral group of order 12 consisting of the cyclic permutations of the 6 vertices and the 6 reflections, as in Figure \ref{alpha_G example}. 
    For $H=\Aut P$, $P^H$ is the origin, so by Theorem \ref{alpha_k,Gformula}, $\alpha_{k,G(H)}=\alpha_{G(H)}={1}$. 
    When $H$ is the trivial group then $P^H=P$ and 
    Theorem \ref{alpha_k,Gformula} gives 
    $\alpha_{k,G(H)}=\alpha_{G(H)}=\frac{1}{2}$, 
    with minimum achieved at any vertex of $P$.
    When $H$ is the subgroup $\ZZ_2$ given by reflection about $y=x$,
    $P^H$ is the line segment 
    $\co\{(-\frac{1}{2},-\frac{1}{2}),(\frac{1}{2},\frac{1}{2})\}$
    and
    $\alpha_{k,G(H)}=\alpha_{G(H)}=\frac{1}{2}$, 
    with minimum achieved at $\pm(\frac{1}{2},\frac{1}{2})$.
    When $H$ is the subgroup $\ZZ_2$ given by reflection about $y=-x$,
    $P^H=\co\{(-1,1),(1,-1)\}$, and
    $\alpha_{k,G(H)}=\alpha_{G(H)}=\frac{1}{2}$, 
    with minimum achieved at $(\pm1,\mp1)$.
    Same for the other reflections.
    When $H$ contains a cyclic permutation, or more than one reflection, $P^H$ is the origin and $\alpha_{k,G(H)}=\alpha_{G(H)}={1}$.
\end{example}

\begin{example}
\lb{P1P1Exam}
    Let $P$ be the polytope for $\PP^1\times\PP^1$,
    i.e., with $v_1,v_2$ as above and $v_3=(-1,0)$ and $v_4=(0,-1)$, so
    $P=[-1,1]^2$ and 
    $\Aut P\cong D_{8}$, the dihedral group of order 8 generated by 4 reflections
    shown in Figure \ref{alpha_G example}. 
    Similar to previous examples, we obtain
    $\alpha_{k,G(H)}=\alpha_{G(H)}={1}$ whenever $H$ contains a cyclic permutation or more than one reflection. Otherwise it equals $1/2$.
\end{example}

\begin{example}
\lb{alphakmGHExample}
Let $P$ be the polytope for $X=\PP^2$ from Example \ref{P2ExampleFirstExam}.
Recall that $\Aut P =S_3$. Let $H=A_3$, the alternating group of order 3, generated by a cyclic permutation of the three vertices. 
The group is represented by $\{I,A,A^2\}$ where
$$
A=\begin{pmatrix}-1&-1\cr1&0\end{pmatrix}.
$$
All orbits of $H$ on $M\cong\ZZ^2$ have cardinality 3 except the orbit of the origin, that has cardinality 1. Any $G(H)$-invariant subspace of $H^0(X,-kK_X)$
is spanned by a collection of monomials indexed by a union
$\cup_{i=1}^\ell O_i^{(k)}$
of $H$-orbits on $kP\cap M$. This forces the subspace to have dimension $m=3\ell$ or $m=3(\ell-1)+1$.
Thus, $m\equiv_30,1$ 
(see Figure \ref{P2-counterexample-figure}
for the case $k=1$). As a result, if $m\equiv_32$ 
then $\alpha_{k,m,G(H)}=\infty$.
For this reason we only consider in \S\ref{alphakmSec} the case $H=\{\id\}$.
\end{example}

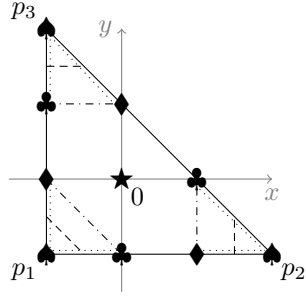
\begin{figure}
        \centering
        \begin{tikzpicture}
            \draw[->,gray](-1.5,0)--(2,0)node[below]{$x$};
            \draw[->,gray](0,-1.5)--(0,2)node[left]{$y$};
            \draw(0,0)node[below right]{$0$};
            \draw(2,-1)node[below right]{$p_2$};
            \draw(-1,2)node[above left]{$p_3$};
            \draw(-1,-1)node[below left]{$p_1$};
            \draw(2,-1)node{$\spadesuit$}--(1,0)node{$\clubsuit$}--(0,1)node{$\blacklozenge$}--(-1,2)node{$\spadesuit$}--(-1,1)node{$\clubsuit$}--(-1,0)node{$\blacklozenge$}--(-1,-1)node{$\spadesuit$}--(0,-1)node{$\clubsuit$}--(1,-1)node{$\blacklozenge$}--cycle;
            \draw(0,0)node{$\bigstar$};
            \draw[dashdotted](1,-1)--(1,0);
            \draw[dashdotted](-1,1)--(0,1);
            \draw[dashdotted](-1,0)--(0,-1);
            \draw[dotted](-1,-0.95)--(0,-0.95);
            \draw[dotted](-0.95,-1)--(-0.95,0);
            \draw[dotted](-0.95,1)--(-0.95,2);
            \draw[dotted](-0.075,0.975)--(-0.95,1.90);
            \draw[dotted](1,-0.95)--(2,-0.95);
            \draw[dotted](0.975,-0.075)--(1.90,-0.95);
            \draw[densely dashed](1.5,-1)--(1.5,-0.5);
            \draw[densely dashed](-1,1.5)--(-0.5,1.5);
            \draw[densely dashed](-1,-0.5)--(-0.5,-1);
        \end{tikzpicture}
        \caption{The four orbits $O_1^{(1)},\ldots,O_4^{(1)}$ of the action of the cyclic group of order 3 (generated by cyclic permutation of the homogeneous coordinates of $\PP^2$) on the polytope corresponding to $\PP^2$ with $k=1$.
        The dashed-dotted lines depict the sets on which $h_{\Delta_1}$ 
        is constant equal to 1 and that compute $\alpha_{1,2,(S^1)^2}=1/2$.
        These sets are the three components of $-\partial P\cap P$.
        The dotted lines depict the sets on which $h_{\Delta_1}$ 
        is not constant but whose minimum is 1 and that also compute $\alpha_{1,2,(S^1)^2}=1/2$.
        The dashed lines depict the sets on which $h_{\Delta_1}$ 
        is constant equal to $3/2$ and that compute $\alpha_{2,2,(S^1)^2}=2/5$.
        These sets are the three components of $-\frac32\partial P\cap P$.}
        \lb{P2-counterexample-figure}
    \end{figure}

\begin{example}
\lb{alphakmTExample}
Let $P$ be the polytope for $X=\PP^2$ from Example \ref{P2ExampleFirstExam}.
We compute the obstruction \eqref{starPEq}, i.e., compute
$
\hbox{\rm argmax}_{P}\|\,\cdot\,\|_{-P}
=
\hbox{\rm argmax}_{P}h_{Q}
=
\hbox{\rm argmax}_{P}h_{\Delta_1}
$.
Recall $\Delta_1=\{v_1,v_2,v_3\}=\{(1,0),(0,1),(-1,-1)\}$.
Let $y=(y_1,y_2)$ be coordinates on $P$.
Then,
$$
h_{v_1}(y)=y_1,
\q
h_{v_2}(y)=y_2,
\q
h_{v_3}(y)=-y_1-y_2,
$$
so 
$$
h_{\Delta_1}(y)=\max\{y_1,y_2,-y_1-y_2\}.
$$
By Lemma \ref{ConvMaxAttaindVertexLem} $\Ver P\cap \hbox{\rm argmax}_{P}h_{\Delta_1}
\not=\emptyset$. Note,
$$
\Ver P=\{p_1,p_2,p_3\}=\{(-1,-1),(2,-1),(-1,2)\}.
$$
So
$$
\max_P h_Q=\max\{h_Q(p_1),h_Q(p_2),h_Q(p_3)\}=\max\{2,2,2\}=2,
$$
and $\Ver P\subset \hbox{\rm argmax}_{P}h_{\Delta_1}$.
Suppose that $h_Q(y)=2$. Then a computation shows that $y\in\Ver P$.
Thus $\Ver P= \hbox{\rm argmax}_{P}h_{\Delta_1}$. By Theorem \ref{SecondMainThm},
$\alpha_{k,m,(S^1)^n}>\alpha=\frac13$ for all $k\in\NN$ and $m\ge2$.

Let us compute $\alpha_{1,2,(S^1)^n}$.
Let $\cF=\{f_1,f_2\}\subset P\cap M$. 
By Corollary \ref{alphakmtorusCor} it suffices to consider `minimal' $\cF$ in the sense
that for no other $\cF'$, $\co\cF'\subset\co\cF$.
So for instance, we do not need to consider 
$\{p_1,p_2\}$ but instead
$$
\{p_1,(0,-1)\}, 
\{(0,-1),(1,-1)\},
\{(1,-1),p_2\}.
$$
In fact,
$\min_{\co\{p_1,(0,-1)\}}h_Q=h_Q(0,-1)=1$ while
$\min_{\{p_1,p_2\}}h_Q=h_Q(1/2,-1)=1/2$.

Also, again by Corollary \ref{alphakmtorusCor}, 
we do not need to consider any $\cF$ such that $0\in\co\cF$
since on such an interval the support function attains its absolute
minimum, i.e., vanishes.

Next, observe that 
$
h_{\Delta_1}=1
$
on the vertices of the hexagon
$$
P_1:=\co\{\pm(1,0),\pm(0,1),\pm(-1,1)\}.
$$
By Lemma \ref{ConvMaxAttaindVertexLem},
$$
h_{\Delta_1}|_{P_1}\le h_{\Delta_1}|_{\Ver P_1}=1.
$$
Thus (by Corollary \ref{alphakmtorusCor})
we do not need to consider any $\cF$ 
such that $\co\cF$ intersects $P_1$,
as such an $\cF$ will satisfy
$\min_{\co\cF}h_Q\le 1$.
Since every $\cF$ intersects $P_1$ it follows that
$\alpha_{1,2,(S^1)^2}=1/2$
by Corollary \ref{alphakmtorusCor}.
See Figure \ref{P2-counterexample-figure}.

By the same reasoning we can compute $\alpha_{k,2,(S^1)^2}$ using
the hexagons
$$
\begin{aligned}
P_k:=&\co\bigg\{
\Big(2-\frac1k,\frac1k-1\Big),
\Big(2-\frac1k,-1\Big),
\Big(-1+\frac1k,-1\Big),
\Big(-1,-1+\frac1k\Big),
\cr
&\q\q\q\q
\Big(-1,2-\frac1k\Big),
\Big(-1+\frac1k,2-\frac1k\Big)
\bigg\}
\cr
=&\frac1k\co\bigg\{
\Big(2 k-1,1-k\Big),
\Big(2k-1,-k\Big),
\Big(-k+1,-k\Big),
\Big(-k,-k+1\Big),
\cr
&\q\q\q\q\Big(-k,2k-1\Big),
\Big(-k+1,2k-1\Big)
\bigg\}\subset P\cap M/k.
\end{aligned}
$$
By Lemma \ref{ConvMaxAttaindVertexLem},
$$
h_{\Delta_1}|_{P_k}\le h_{\Delta_1}|_{\Ver P_k}=2-1/k.
$$
Since every $\cF$ intersects $P_k$ it follows that
$$\alpha_{k,2,(S^1)^2}=\Big[1+2-\frac1k \Big]^{-1}=\frac{k}{3k-1}$$
by Corollary \ref{alphakmtorusCor}. Note that
$\lim_k\frac{k}{3k-1}=1/3=\alpha$ (recall Example \ref{P2ExampleFirstExam}
and \eqref{alphakBJEq}) in accordance with Proposition \ref{convergencalphakmProp}
and that this is a monotone sequence in accordance with
Proposition \ref{alpha_k,m,G monotone in m}.

\end{example}

The next result should be well-known but we could not find a reference.
\blem
\lb{alphahalfLem}
    Let $X$ be toric Fano.
    Let $P\subset M_\RR$ (see (\ref{NRMREq}), (\ref{PQcircDef})) be 
the polytope associated to $(X,-K_X)$. Then $\alpha\le1/2$, and
$P$ is centrally symmetric
(i.e., $P=-P$) if and only if $\alpha=1/2$.
If $P$ is not centrally symmetric then $\alpha\le 1/3$.
\elem

\begin{proof}
    If $P\neq-P$, then there is $u_0\in P\setminus-P$. 
    By \eqref{PDef2Eq},
$-P=\{y\in M_\RR\,:\, \max_j\langle v_j,y\rangle\le1\}$.
Thus 
    $\langle u_0,v_{i_0}\rangle>1$
    for some $v_{i_0}\in\Delta_1\subset N$ (recall \eqref{Delta1Eq}). Recalling Theorem \ref{alpha_k,Gformula} and \eqref{alphakBJEq}, or \cite[Corollary 7.16]{BJ20},
    $$
        \alpha=\min_{u\in P}\min_i\frac{1}{1+
        \langle u,v_i\rangle}\leq\frac{1}{1+
        \langle u_0,v_{i_0}\rangle}<\frac{1}{2}.
    $$
In fact, by convexity one may take $u_0\in \Ver P\setminus-P$,
so $u_0\in M$ by Lemma \ref{PIntegralLem} and hence 
$\ZZ\ni \langle u_0,v_{i_0}\rangle\ge2$, i.e., $\alpha\le1/3$.

    If $P=-P$, by \eqref{anotheralphakconstEq} and \eqref{alphakBJEq},
    $$
        \alpha=\frac{1}{1+\max\limits_{ P}\left\|\,\cdot\,\right\|_{-P}}=\frac{1}{1+\max\limits_{ P}\left\|\,\cdot\,\right\|_P}=\frac{1}{1+1}=\frac{1}{2}
    $$
by \eqref{normdef} and \eqref{infachievedEq}.
\end{proof}

\begin{figure}
    \centering
    \begin{subfigure}{0.4\textwidth}
        \centering
        \begin{tikzpicture}
            \draw[->,gray](-2,0)--(1.5,0)node[below]{$x$};
            \draw[->,gray](0,-2)--(0,1.5)node[left]{$y$};
            \draw(0,0)node[below right]{$0$};
            \draw(-.5,0)--(0,-.5)--(1,-.5)--(-.5,1)--cycle;
            \draw[dashed](1,0)--(0,1)--(-2,1)--(1,-2)--cycle;
        \end{tikzpicture}
    \end{subfigure}
    \begin{subfigure}{0.4\textwidth}
        \centering
        \begin{tikzpicture}
            \draw[->,gray](-1.75,0)--(1.75,0)node[below]{$x$};
            \draw[->,gray](0,-1.75)--(0,1.75)node[left]{$y$};
            \draw(0,0)node[below right]{$0$};
            \draw(-.5,.5)--(-.5,-.5)--(.5,-.5)--(.5,0)--(0,.5)--cycle;
            \draw[dashed](1,-1)--(1,1)--(-1,1)--(-1,0)--(0,-1)--cycle;
        \end{tikzpicture}
    \end{subfigure}
    \caption{The polytope $P$ for $\PP^2$ blown up 1 or 2 points, and the level set $\{\|\cdot\|_{-P}=2\}$.}\label{counterexample figure}
\end{figure}
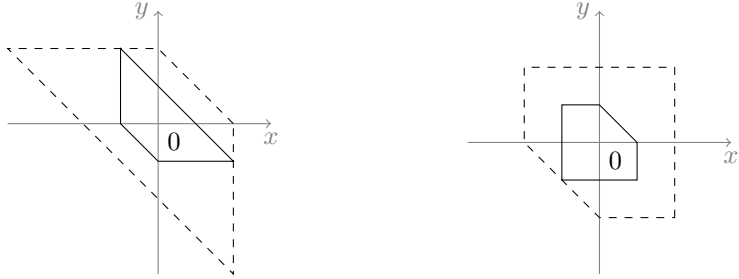

\begin{example}
\lb{alphakmSymmExample}
The purpose of this paragraph is to show that when $n=2$, condition \eqref{starPEq} holds for a Delzant polytope $P$
if and only if $P$ is not centrally symmetric, i.e., $P\not=-P$. 
Alternatively, in $n=2$, Conjecture \ref{TianGenConj}
holds if and only if $\alpha=1/2$, by Lemma \ref{alphahalfLem}.
However, in dimension $n\ge3$ this is no longer the case: for instance 
$\PP^2\times\PP^1$ has a non-centrally symmetric polytope $P$
but condition \eqref{starPEq} fails.

When $n=2$, the centrally symmetric
Delzant polytopes are associated to $\PP^2$ blown-up at the three
non-colinear points, and  $\PP^1\times\PP^1$
 (Examples \ref{P2blowupthreeExam}--\ref{P1P1Exam}).
 Recall the description of \eqref{starPEq} given in \S\ref{ProofThm1.6SubSec}.
Since $P=-P$, the intersection of $P$ and
$\max_P\|\cdot\|_{-P}\partial(-P)$ is precisely all of $\partial P$, i.e.,
\eqref{starPEq} fails. 

Thus, it remains to check $\PP^2$ blown-up at up to two points. We have already
showed that \eqref{starPEq} holds for (the non-centrally symmetric) $P$ coming from  
$\PP^2$ (Example \ref{alphakmTExample}).
It thus remains to check the remaining two cases.
The polytope for the 1-point blow-up is a subset of the one for $\PP^2$ by 
chopping a corner. Thus 
the intersection of $P$ and
$\max_P\|\cdot\|_{-P}\partial(-P)=-2\partial P$
is still just the two vertices $(-1,2)$ and $(2,-1)$.
For the 2-point blow-up 
the intersection of $P$ and
$\max_P\|\cdot\|_{-P}\partial(-P)=-2\partial P$
is just the vertex $(-1,-1)$, concluding the proof. See Figure \ref{counterexample figure}.

\end{example}

\bigskip
\textsc{University of Maryland}

\bigskip
{\tt cjin123@terpmail.umd.edu, yanir@alum.mit.edu}

\end{document}